\documentclass[10pt, final]{amsart}
\usepackage{amsmath}
\usepackage[hidelinks]{hyperref}
\usepackage[all,cmtip]{xy}
\usepackage{mathrsfs}
\usepackage{float}
\usepackage{subfigure}
\usepackage{xspace}
\usepackage{enumitem}
\usepackage{mathtools}
\usepackage[backend=biber,
            style=alphabetic,
            sorting=nyt,
            uniquename=false,
            giveninits=true,
            doi=true,url=false]{biblatex}

\newtheorem{theorem}{Theorem}[section]
\newtheorem{corollary}[theorem]{Corollary}
\newtheorem{lemma}[theorem]{Lemma}
\newtheorem{proposition}[theorem]{Proposition}
\theoremstyle{definition}
\newtheorem{definition}[theorem]{Definition}
\newtheorem{remark}[theorem]{Remark}
\newtheorem{example}[theorem]{Example}

\newcommand{\goodgap}{\hspace{\subfigtopskip}\hspace{\subfigbottomskip}}
\newcommand{\lra}{\longrightarrow}
\newcommand{\ff}{\mathbb F}
\providecommand{\R}{\mathbb{R}}
\providecommand{\N}{\mathbb{N}}

\newcommand{\pr}{\operatorname{pr}}
\providecommand{\calM}{\mathcal M}
\providecommand{\calO}{\mathcal O}
\providecommand{\calW}{\mathcal W}
\newcommand{\frakg}{\mathfrak g}

\newcommand{\bF}{\mathbf F}
\providecommand{\abs}[1]{\ensuremath{\left\lvert{#1}\right\rvert}}
\providecommand{\ip}[2]{\ensuremath{\left\langle{#1},{#2}\right\rangle}}
\providecommand{\norm}[1]{\lVert#1\rVert}

\newcommand{\dert}{\frac{d}{dt} \bigg|_{t=0}}
\providecommand{\CE}{\ensuremath{\coloneqq}\xspace}
\providecommand{\jdef}[1]{\index{#1}\emph{#1}}
\providecommand{\stext}[1]{\ensuremath{\quad\text{#1}\quad}}
\providecommand{\pd}[2]{{\frac{\partial {#1}}{\partial {#2}}}}
\providecommand{\conj}{\overline}
\providecommand{\FlowC}[2]{\ensuremath{\bF^{{#1}}_{{#2}}}}  
\providecommand{\FlowCBV}[3]{\ensuremath{\bF^{{#1}}_{{#2},{#3}}}}  
\providecommand{\FlowD}[1]{\ensuremath{\bF^{{#1}}}} 
\providecommand{\ti}[1]{\ensuremath{\tilde{{#1}}}}

\renewbibmacro{in:}{} 
\numberwithin{equation}{section}

\begin{document}

\title[Forced discrete Hamiltonian systems]{Variational integrators
  using\\forced discrete Hamiltonian systems}

\author{Mat\'ias I. Caruso}
\address{\textnormal{(M. I. Caruso)} Depto. de Matem\'atica\\Instituto Balseiro \\ Universidad Nacional de Cuyo - C.N.E.A.\\ Av. Bustillo 9500 \\ San Carlos de Bariloche \\ R8402AGP \\ Argentina \hfill \break
	\indent Consejo Nacional de Investigaciones Cient{\'\i}ficas y T\'ecnicas (CONICET)}
\email{matias.caruso@ib.edu.ar}

\author{Javier Fern\'andez}
\address{\textnormal{(J. Fern\'andez)} Depto. de Matem\'atica\\Instituto Balseiro \\ Universidad Nacional de Cuyo - C.N.E.A.\\ Av. Bustillo 9500 \\ San Carlos de Bariloche \\ R8402AGP \\ Argentina}
\email{jfernand@ib.edu.ar}

\author{Cora Tori}
\address{\textnormal{(C. Tori)} Depto. de Ciencias B\'asicas \\ Facultad de Ingenier\'ia \\ Universidad Nacional de La Plata.
	Calle 116 entre 47 y 48 \\ La Plata \\ Buenos Aires \\ 1900 \\ Argentina \hfill \break
	\indent Centro de Matem\'atica de La Plata (CMaLP)}
\email{cora.tori@ing.unlp.edu.ar}

\author{Marcela Zuccalli}
\address{\textnormal{(M. Zuccalli)} Depto. de Matem\'atica \\ Facultad de Ciencias Exactas \\ Universidad Nacional de La Plata.
	Calles 50 y 115 \\ La Plata \\ Buenos Aires \\ 1900 \\ Argentina \hfill \break
	\indent Centro de Matem\'atica de La Plata (CMaLP)}
\email{marce@mate.unlp.edu.ar}


\begin{abstract}
  We study discrete Hamiltonian systems defined on cotangent bundles
  that are subjected to external forces, whose trajectories are
  determined by a discrete variational principle. We analyze the
  evolution of the canonical symplectic structure and, when a Lie
  group of symmetries is present, the corresponding evolution of the
  associated momenta. Given a continuous forced Hamiltonian system, we
  construct an exact discrete analogue whose order-$r$ approximations
  yield trajectories that approximate the continuous ones with
  accuracy of at least order $r$. We also give two methods to build
  approximate discrete systems. Combining these, we obtain a
  variational integrator: first approximate the exact discrete system
  and then solve the resulting algebraic equations of motion.
\end{abstract}

\maketitle

\tableofcontents


\section{Introduction}
\label{sec:introduction}

In Numerical Analysis, Geometric Numerical Integration refers to the
construction of algorithms that approximate the solution of ordinary
differential equations while preserving the geometric characteristics
of the given
problem~\cite{bo:hairer_lubich_wanner-geometric_numerical_integration}. When
the differential equations arise as equations of motion of mechanical
systems there is a well known way of constructing geometric
integrators using what are known as Discrete Mechanical Systems. The
solution of the equation of motion of a (continuous) mechanical
system, a trajectory, can be seen as a critical point of a variational
problem in a space of paths. Similarly, trajectories of a
\emph{discrete} mechanical system are critical points of a certain
functional ---defined in a space of discrete paths that, crucially, is
finite dimensional--- and are characterized by equations of motion
that are algebraic. Solving these equations gives rise to a numerical
integrator of the original differential problem
(see~\cite{ar:marsden_west:2001:discrete_mechanics_and_variational_integrators}),
known as a \jdef{variational integrator}, provided that
\begin{enumerate}
\item \label{it:goal-approximation} the discrete mechanical system is
``close enough'' to the continuous one, and
\item \label{it:goal-close_implies_flow_close}
  that~\eqref{it:goal-approximation} suffices to conclude that the
  trajectories of the systems are ``close enough''.
\end{enumerate}
Also of importance,
\begin{enumerate}[resume]
\item\label{it:goal-properties} how well do variational integrators
  preserve the geometric characteristics of the system?
\end{enumerate}

The description of trajectories of mechanical systems defined on a
configuration space $Q$ as critical points of a functional is
characteristic of the Lagrangian formulation of Mechanics
---variational formulation would be a better name---, where the
functional, called the action, is computed using a Lagrangian function
$L$ over the tangent bundle $TQ$. Alternatively, it is possible to
give a characterization of trajectories as critical points of a
functional on curves in the cotangent bundle $T^*Q$, computed using a
Hamiltonian function $H$ on $T^*Q$ (see~\cite{bo:AM-mechanics}
or~\cite{bo:goldstein:classical_mechanics}). In most cases, the two
descriptions are equivalent. Both approaches have discrete versions:
by far, the most common is the Lagrangian approach, where the discrete
action is defined using a discrete Lagrangian function
$L_d:Q\times Q\rightarrow \R$
(see~\cite{ar:wendtlandt_marsden-mechanical_integrators_derived_from_a_discrete_variational_principle},
~\cite{ar:marsden_west:2001:discrete_mechanics_and_variational_integrators}
and~\cite{ar:marrero_martin_martinez-discrete_lagrangian_and_hamiltonian_mechanics_on_lie_groupoids}). The
discrete Hamiltonian approach considers discrete paths in $T^*Q$ and
the corresponding action functional is constructed using a discrete
Hamiltonian function $H_d:T^*Q\rightarrow \R$
(see~\cite{ar:lall_west-discrete_variational_hamiltonian_mechanics}
and~\cite{ar:leok_zhang:2011:discrete_hamiltonian_variational_integrators}). These
variational integrators are known to satisfy
points~\eqref{it:goal-approximation}
and~\eqref{it:goal-close_implies_flow_close} and,
regarding~\eqref{it:goal-properties}, preserve some natural symplectic
structures whereas, if symmetry is present, the associated momenta are
conserved
(see~\cite{ar:marsden_west:2001:discrete_mechanics_and_variational_integrators},~\cite{ar:fernandez_graiffZurita_grillo:2021:error_analysis_of_forced_discrete_mechanical_systems}
and~\cite{ar:schmitt_leok:2017:properties_of_hamiltonian_variational_integrators}).

Unfortunately, in many real world applications, it is necessary to
consider mechanical systems that are subjected to external
forces. These systems rarely have the nice conservation properties
mentioned above. Still, the forced discrete Lagrangian systems have
been successfully used and studied for some time
(see~\cite{ar:marsden_west:2001:discrete_mechanics_and_variational_integrators},
\cite{ar:deDiego_deAlmagro-variational_order_for_forced_lagrangian_systems},
\cite{ar:fernandez_graiffZurita_grillo:2021:error_analysis_of_forced_discrete_mechanical_systems},
and~\cite{ar:caruso_fernandez_tori_zuccalli:2023:lagrangian_reduction_of_forced_discrete_mechanical_systems}).

On the other hand, the study of forced discrete Hamiltonian systems
lags behind. The purpose of this paper is to study such systems, with
focus on the case where the configuration manifold $Q$ is a finite
dimensional real vector space. We define trajectories of a forced
discrete Hamiltonian system (FDHS) as the solutions of a discrete
variational problem, similar to what is done in Section 3.2
of~\cite{ar:deLeon_lainz_lopezGordon-discrete_hamilton_jacobi_theory_for_systems_with_external_forces},
and, also, extending the idea used in Section 3
of~\cite{ar:leok_zhang:2011:discrete_hamiltonian_variational_integrators}
to the forced case. Actually, for a technical reason, we introduce a
slightly more general notion that we call \jdef{extended trajectory}
of the system.  Trajectories, extended or plain, are, also,
characterized as solutions of a system of algebraic equations. As in
the Lagrangian case, we introduce a notion of forced discrete Legendre
transformation ---in fact, two of them, a $+$ and a $-$ version--- and
call regular the FDHSs where they are local diffeomorphisms. We prove
that, if the space of extended trajectories of length $N$ is
non-empty, it is an open set in $Q^{N+1}$.

We also study some of the structural properties of FDHSs. As expected,
we see that the canonical symplectic form $\omega_Q$ on $T^*Q$ is not,
in general, preserved by the flow; an exception is the case when the
forces are closed. Similarly, if a Lie group is a symmetry group of
the system, we obtain a formula describing the evolution of the
corresponding canonical momentum map and give a condition that ensures
that it is conserved. These properties are the current, partial,
answer to the point~\eqref{it:goal-properties} above.

Even though the theory that we develop in this paper is purely
Hamiltonian, we see that given a ``good'' Lagrangian system (meaning
hyperregular and satisfying a certain ``modified hyperregularity
condition'') it is possible to construct an FDHS such that there is a
bijective correspondence between the trajectories of the former system
and the extended trajectories of the latter.

A central part of the paper is the error analysis of the variational
integrators constructed using FDHSs. We first prove that, given a
(continuous) forced Hamiltonian system $\calM$, there is an
$h$-dependent family of FDHSs $\calM_{d,h}^E$ ($h$ is a scalar
parameter defined in a neighborhood of $0$) such that the trajectories
of $\calM_{d,h}^E$ are the trajectories of $\calM$ evaluated at times
that are multiples of $h$, and conversely. We call $\calM_{d,h}^E$ the
\emph{exact} FDHS associated to $\calM$. Unfortunately, this system
$\calM_{d,h}^E$ has no direct practical use because it cannot be
computed effectively in most real cases. The true interest in
$\calM_{d,h}^E$ comes from using approximations $\calM_{d,h}$ of
$\calM_{d,h}^E$, usually thought of as discretizations of $\calM$. The
main result we prove in this area is that if
$\calM_{d,h} = \calM_{d,h}^E + \calO(h^{r+1})$ for some $r\in\N$
(notation to be explained), then the corresponding flows satisfy
$\FlowC{\calM}{h} = \FlowD{\calM_{d,h}^E} = \FlowD{\calM_{d,h}} +
\calO(h^{r+1})$. Thus, if we choose a discretization of $\calM$ that
is accurate to order $r$, the corresponding variational integrator
has, at least, the same order of accuracy (local error of the same
order). This is the answer to the
point~\eqref{it:goal-close_implies_flow_close}.

We also touch on the practical matter of constructing discretizations
for a given (continuous) forced Hamiltonian system $\calM$. We propose
a method based on expanding $\calM_{d,h}^E$ as a Taylor series around
$h=0$; we provide explicit formulas for the expansion up to orders $1$
and $2$. An alternative method using Gaussian quadrature and the
shooting method for integrating boundary value problems for ODEs is
discussed. Anyone of these methods provides a concrete way to
satisfy~\eqref{it:goal-approximation}.

The paper is structured as follows:
Section~\ref{sec:forced_hamiltonian_systems_variational_approach}
reviews some notions and results on (continuous) forced Hamiltonian
systems from the variational point of view. Forced discrete
Hamiltonian systems and their dynamics are introduced in
Section~\ref{sec:forced_discrete_Hamiltonian_systems}. Section~\ref{sec:structural_properties}
analyzes the evolution of the symplectic form and momentum by the flow
of an FDHS; it also discusses the construction of an FDHS from a given
forced Lagrangian system, such that the trajectories of the two
systems are in bijective
correspondence. Sections~\ref{sec:error_analysis-discrete_exact_systems}
and~\ref{sec:error_analysis-approximations} are dedicated to the error
analysis of the corresponding variational integrators: the former
focuses on the exact FDHS while the latter proves that a
discretization of order $r$ leads to an integrator of order, at least,
$r$. Last, Section~\ref{sec:construction_of_FDHS} discusses two
methods that can be used to construct FDHSs in practice.


\section{Forced Hamiltonian systems: variational approach}
\label{sec:forced_hamiltonian_systems_variational_approach}

Since we are looking for a variational version of forced discrete
Hamiltonian system, we will work, following
\cite{ar:leok_zhang:2011:discrete_hamiltonian_variational_integrators},
on a real vector space $Q$. Hence, its cotangent bundle can be
trivialized as $T^*Q \simeq Q \times Q^*$ and expressions such as
$p \dot{q}$ or $(q,p)$ for a curve in it are adequate.

We begin with a brief review of the variational formulation of
Hamiltonian mechanics.

\begin{definition}
  A \jdef{forced Hamiltonian system} is a triple $(Q,H,\phi)$, where
  $Q$ is a finite dimensional real vector space, $H : T^*Q \lra \R$ is
  a smooth function and $\phi \in \Omega^1(T^*Q)$ is a horizontal
  $1$-form.
	
  A curve $(q,p) : \R \lra T^*Q$ is a \jdef{(type I) trajectory} of
  $(Q,H,\phi)$ if it satisfies
  \begin{equation*}
    \delta \int_0^T (p \dot{q} - H(q,p)) \ dt + \int_0^T
    \check{\phi}(q,p)(\delta q) \ dt = 0,
  \end{equation*}
  for all infinitesimal variation $(\delta q,\delta p)$ over $(q,p)$
  such that $\delta q(0) = 0$ and $\delta q(T) = 0$, where
  $\check{\phi}$ is given by
  $\phi(q,p)(\delta q,\delta p) = \check{\phi}(q,p)(\delta q)$ (recall
  that $\phi$ is horizontal). This kind of infinitesimal variations
  will be called of type I.
\end{definition}

\begin{proposition}\label{prop:Hamilton_equations}
  Let $\calM \CE  (Q,H,\phi)$ be a forced Hamiltonian system. A curve
  $(q,p)$ on $T^*Q$ is a trajectory of $\calM$ if and only if it
  satisfies
  \begin{equation}\label{eq:Hamilton_equations}
    \begin{cases}
      \dot{q}(t) = D_2 H(q(t),p(t)) \\
      \dot{p}(t) = -D_1 H(q(t),p(t)) + \check{\phi}(q(t),p(t)).
    \end{cases}
  \end{equation}
\end{proposition}

\begin{proof}
  Let $(q,p)$ be a curve in $T^*Q$ and $(\delta q,\delta p)$ an
  infinitesimal variation over $(q,p)$. Then, the standard integration
  by parts argument leads to
  \begin{equation}\label{eq:Hamilton_equations-1}
    \begin{split}
      \delta \int_0^T (p \dot{q} -& H(q,p)) \ dt + \int_0^T \check{\phi}(q,p)(\delta q) \ dt = p(T) \delta q(T) - p(0) \delta q(0) \\& + \int_0^T \left( \left( - \dot{p} - D_1H(q,p) + \check{\phi}(q,p) \right) \, \delta q + \left( \dot{q} - D_2H(q,p) \right) \, \delta p \right) \ dt.
    \end{split}
  \end{equation}
	
  If $(q,p)$ is a trajectory of $\calM$ and $(\delta q,\delta p)$ has
  endpoints of type I, then, since the variations are arbitrary,
  $-\dot{p} - D_1H(q,p) + \check{\phi}(q,p) = 0$ and
  $\dot{q} - D_2H(q,p) = 0$, proving~\eqref{eq:Hamilton_equations}.
	
  Conversely, if $(q,p)$ satisfies \eqref{eq:Hamilton_equations} and
  $(\delta q,\delta p)$ has type I, then
  \eqref{eq:Hamilton_equations-1} yields
  \begin{equation*}
    \delta \int_0^T (p \dot{q} - H(q,p)) \ dt + \int_0^T
    \check{\phi}(q,p)(\delta q) \ dt = 0,    
  \end{equation*}
  that is, $(q,p)$ is a trajectory of $\calM$.
\end{proof}

The equations \eqref{eq:Hamilton_equations} are called \jdef{forced
  Hamilton equations} (see, for example, \cite[Section
3.1.2]{ar:marsden_west:2001:discrete_mechanics_and_variational_integrators}).

\begin{example}\label{ex:simple_with_kappa_and_nu}
  Let $Q \CE \R^n$ equipped with the canonical inner product,
  \begin{equation*}
    H(q,p) \CE  \frac{\norm{p}^2}{2m} + V(q) \stext{ and } 
    \phi(q,p) \CE  a(q,p)\, dq ,
  \end{equation*}
  for constant $m > 0$ and a smooth function $a : T^*Q \lra \R^n$. We
  call such systems of \jdef{mechanical type}.
  
  In this case, the equations~\eqref{eq:Hamilton_equations} become
  \begin{equation}\label{ex:simple_with_kappa_and_nu-equations}
    \begin{cases}
      \dot{q}(t) = \frac{1}{m} p(t) \\
      \dot{p}(t) = -\nabla V(q(t)) + a(q(t),p(t)).
    \end{cases}
  \end{equation}
  Solving these equations for $V(q)\CE \nu q$ and $a(q,p) \CE \kappa$
  with constants $\nu, \kappa\in\R^n$, we find that, for initial
  conditions $(q_0,p_0)$, we have $\dot{q}_0 \CE \frac{p_0}{m}$, and
  the trajectory of the system is
  \begin{equation*}
    (q(t),p(t)) = \left( \frac{1}{2m} (\kappa - \nu) t^2 + \frac{p_0}{m} t +
      q_0 , (\kappa - \nu) t + p_0 \right).
  \end{equation*}
\end{example}

The previous formulation works for trajectories whose initial and
final positions are known. There are other formulations in which the
known data are $q(0)$ and $p(T)$ or $p(0)$ and $q(T)$. In this
direction, and inspired by
\cite{ar:leok_zhang:2011:discrete_hamiltonian_variational_integrators},
we have the following definition.

\begin{definition}\label{def:type_II_Hamilton_trajectories}
  Let $\calM \CE (Q,H,\phi)$ be a forced Hamiltonian system. A curve
  $(q,p)$ on $T^*Q$ is a \jdef{type II trajectory} of $\calM$ if it
  satisfies
  \begin{equation*}
    \delta \left( p(T) q(T) - \int_0^T (p \dot{q} - H(q,p)) \ dt \right) -
    \int_0^T \check{\phi}(q,p)(\delta q) \ dt = 0
  \end{equation*}
  for all infinitesimal variations $(\delta q,\delta p)$ over $(q,p)$
  such that $\delta q(0) = 0$ and $\delta p(T) = 0$.
\end{definition}

\begin{proposition}
  Let $\calM \CE (Q,H,\phi)$ be a forced Hamiltonian system. A curve
  $(q,p)$ on $T^*Q$ is a type II trajectory of $\calM$ if and only if
  it satisfies
  \begin{equation*}
    \begin{cases}
      \dot{q}(t) = D_2 H(q(t),p(t)) \\
      \dot{p}(t) = -D_1 H(q(t),p(t)) + \check{\phi}(q(t),p(t)).
    \end{cases}
  \end{equation*}
\end{proposition}

\begin{proof}
  Let $(q,p)$ be a curve in $T^*Q$ and let
  $(\delta q,\delta p)$ be an infinitesimal variation over
  $(q,p)$. Then, by the standard integration by parts argument, we
  have
  \begin{gather*}
    \delta \left( p(T) q(T) - \int_0^T (p \dot{q} - H(q,p)) \ dt \right) -
    \int_0^T \check{\phi}(q,p)(\delta q) \ dt =
    q(T) \delta p(T) + p(0) \delta q(0)\\
    - \int_0^T \left( (-\dot{p} - D_1H(q,p) + \check{\phi}(q,p)) \ \delta q +
      (\dot{q} - D_2H(q,p)) \ \delta p \right) \ dt.
  \end{gather*}
  To complete the proof we just apply the same reasoning as in the
  proof of Proposition~\ref{prop:Hamilton_equations}.
\end{proof}

\begin{remark}\label{rem:bv_continuous_flows}
  The trajectories of the forced Hamiltonian systems are solutions of
  the forced Hamilton equations~\eqref{eq:Hamilton_equations}, so that
  the initial value problem is well studied and it is easy to discuss
  the existence and uniqueness of solutions. In the case of the type
  II trajectories, they are solutions
  of~\eqref{eq:Hamilton_equations}, but with different boundary
  conditions, making the analysis harder. In what follows we will
  assume the existence and uniqueness of these solutions so that we
  can define the \jdef{boundary value flow}
  $\FlowCBV{\calM}{h}{t}(q_0,p_1)$ that assigns to each time $t$ the
  value of the trajectory $(q(t),p(t))$ of the system $\calM$ that
  satisfies $q(0)=q_0$ and $p(h)=p_1$. See, also,
  Proposition~\ref{prop:flow_of_hamilton_equations-b}.
\end{remark}

\begin{example}
  Returning to Example~\ref{ex:simple_with_kappa_and_nu}, we have
  that, given boundary conditions $q(0) \CE q_0$ and $p(h) \CE p_1$,
  the flow is given by
  \begin{equation*}
    \FlowCBV{}{h}{t}(q_0,p_1) = \left( \frac{1}{2m} (\kappa - \nu) t^2 +
      \frac{p_1 - h(\kappa - \nu)}{m} t + q_0 , (\kappa - \nu)(t - h) +
      p_1 \right).
  \end{equation*}
\end{example}

\begin{definition}
  Let $\calM \CE (Q,H,\phi)$ be a forced Hamiltonian system. Its
  \jdef{Legendre transform} is the smooth map ---commuting with the
  projections--- $\ff H : T^*Q \lra TQ$ given by
  \begin{equation*}
    \ff H(q,p) \CE  D_2H(q,p).
  \end{equation*}	
  $\calM$ is said to be \jdef{regular} if $\ff H$ is a local
  diffeomorphism and \jdef{hyperregular} if it is a (global)
  diffeomorphism. 
\end{definition}

\begin{remark}
  Notice that
  $D_2H(q,p) \in Q \times Q^{**} \simeq Q \times Q \simeq TQ$, since
  $Q$ is a vector space.
\end{remark}

We close this section with a result that proves that, locally,
solutions to the forced Hamilton
equations~\eqref{eq:Hamilton_equations} with boundary conditions do
exist. We assume that $Q=\R^n$ or, more generally, $Q\subset \R^n$ is
an open subset. Of course, this means no loss of generality as this
identification corresponds to the choice of a basis in the $\R$-vector
space $Q$. Also, for $x\in\R^n$, we use the norm
$\norm{x}_\infty\CE \max\{\abs{x_1},\ldots, \abs{x_n}\}$ and denote
the corresponding balls by $B^\infty$.

\begin{proposition}\label{prop:flow_of_hamilton_equations-b}
  Let $\calM\CE (Q,H,\phi)$ be a forced Hamiltonian system. For
  $(q_0,p_1)\in Q\times Q^*$, there exist constants $r,r'>0$ and
  $T_-<T_+$ such that $0\in (T_-,T_+)$ and, for any $T\in [T_-,T_+]$
  and
  $(q_0',p_1') \in \conj{B^\infty_{r'}(q_0,p_1)} \subset Q\times Q^*$,
  the boundary value problem
  \begin{equation}\label{eq:flow_of_hamilton_equations-b-II-eqs}
    \begin{cases}
      \dot{q}(t) = D_2H(q(t),p(t)), \stext{ for } T_-<t<T_+\\
      \dot{p}(t) = -D_1H(q(t),p(t)) + \check{\phi}(q(t),p(t))
      \stext{ for } T_-<t<T_+\\
      q(0) = q_0',\stext{ and } p(T)=p_1'
    \end{cases}
  \end{equation}
  has a unique solution $(q(t),p(t))$ that satisfies
  $\norm{(q(0),p(0))-(q_0,p_1)}_\infty < r$. That solution is a smooth
  function of $t$, $T$ and $(q_0',p_1')$, simultaneously.
\end{proposition}

\begin{proof}
  This technical result can be proved starting from a modification of
  the guidance provided in several exercises
  of~\cite{bo:keller-numerical_methods_for_two_point_boundary_value_problems}
  (see Exercises 1.1.4, 1.2.13 and 1.2.14), as well as Theorem 2.10
  of~\cite{bo:teschl-ordinary_differential_equations_and_dynamical_systems}.
\end{proof}

\begin{remark}
  The boundary value
  problem~\eqref{eq:flow_of_hamilton_equations-b-II-eqs} is well
  behaved, even when $T=0$ (where it becomes an initial value problem)
  because the boundary conditions are set on independent functions
  ($q(0)$ vs. $p(T)$). This should be contrasted with the
  corresponding differential problem for (type I) trajectories where
  the boundary conditions are imposed on the same function ($q(0)$ and
  $q(T)$) which leads to singular behavior of the system when $T=0$.
\end{remark}


\section{Forced discrete Hamiltonian systems}
\label{sec:forced_discrete_Hamiltonian_systems}

Given a forced Hamiltonian system $\calM \CE (Q,H,\phi)$ and $h > 0$,
we introduce a notion of discrete Hamiltonian as an approximation
\begin{equation*}
  H_d(q_0,p_1) \approx p(h) q(h) -
  \int_0^h (p(t) \dot{q}(t) - H(q(t),p(t))) \ dt,  
\end{equation*}
where $(q(t),p(t)) \CE \FlowCBV{\calM}{h}{t}(q_0,p_1)$ is the
trajectory of $\calM$ that satisfies the boundary conditions
$q(0) = q_0$ and $p(h) = p_1$. Similarly, the discrete force is an
approximation
\begin{equation*}
  \phi_d(q_0,p_1)(\delta q_0, \delta p_1) \approx
  \int_0^h \phi(q(t),p(t))
  T_{(q_0,p_1)} \FlowCBV{\calM}{h}{t}(\delta q_0, \delta p_1) \ dt.
\end{equation*}

\begin{definition}\label{def:FDHS}
  A \jdef{forced discrete Hamiltonian system} (FDHS) is a triple
  $\calM_d \CE (Q,H_d,\phi_d)$ where $Q$ is a finite dimensional real
  vector space, $H_d : T^*Q \lra \R$ is a smooth function and
  $\phi_d$ is a $1$-form on $T^*Q$.

  A discrete curve $(q_\cdot,p_\cdot) = ((q_0,p_0),\ldots,(q_N,p_N))$ is a
  \jdef{type II trajectory} of $\calM_d$ if it satisfies
  \begin{equation*}
    \delta \left( p_N q_N - \sum_{k=0}^{N-1} (p_{k+1} q_{k+1} - H_d(q_k,p_{k+1}))
    \right) - \sum_{k=0}^{N-1} \phi_d(q_k,p_{k+1})(\delta q_k,\delta p_{k+1}) = 0
  \end{equation*}
  for all infinitesimal variations $(\delta q_\cdot,\delta p_\cdot)$ over
  $(q_\cdot,p_\cdot)$ with boundary conditions $\delta q_0 = 0$ and
  $\delta p_N = 0$.
\end{definition}

In what follows, unless explicitly stated otherwise, we will consider
systems with type II trajectories. Thus, we will drop the ``type II''
in the name.

\begin{example}\label{ex:simple_with_kappa_and_nu-discrete}
  A simple FDHS on $Q\CE \R^n$ that is, somehow, a discretization of
  the system that appears in
  Example~\ref{ex:simple_with_kappa_and_nu}, whose notation we use, is
  given by 
  \begin{gather*}
    H_d(q,p):=pq + h H(q,p) = pq +
    h \left(\frac{\norm{p}^2}{2m} + V(q)\right),\\
    \phi_d(q,p) \CE h\, \phi(q,p) = h\, a(q,p) dq,
  \end{gather*}
  where $h>0$ is a constant.
\end{example}

\begin{proposition}\label{prop:discrete_Hamilton_equations}
  Let $\calM_d \CE (Q,H_d,\phi_d)$ be an FDHS. A discrete curve
  $(q_\cdot,p_\cdot) = ((q_0,p_0),\ldots,(q_N,p_N))$ is a trajectory
  of $\calM_d$ if and only if it satisfies
  \begin{equation}\label{eq:discrete_Hamilton_equations}
    \begin{cases}
      q_k = D_2H_d (q_{k-1},p_k) - \phi_d^p(q_{k-1},p_k), \\
      p_k = D_1H_d (q_k,p_{k+1}) - \phi_d^q(q_k,p_{k+1})
    \end{cases}
    \stext{for} k=1,\ldots,N-1,
  \end{equation}
  where we have used the decomposition associated to the Cartesian
  product,
  \begin{equation}\label{eq:phi_d_qp_decomposition}
    \phi_d(q,p)(\delta q,\delta p) = \phi_d^q(q,p)(\delta q) +
    \phi_d^p(q,p)(\delta p).
  \end{equation}
\end{proposition}

\begin{proof}
  Let $(\delta q_\cdot,\delta p_\cdot)$ be an infinitesimal variation
  over $(q_\cdot,p_\cdot)$. Then,
  \begin{equation}\label{eq:discrete_Hamilton_equations-1}
    \begin{split}
      \delta \bigg( p_N &q_N - \sum_{k=0}^{N-1} (p_{k+1} q_{k+1} - H_d(q_k,p_{k+1})) \bigg) - \sum_{k=0}^{N-1} \phi_d(q_k,p_{k+1})(\delta q_k,\delta p_{k+1}) \\
             =&  \sum_{k=1}^{N-1} \big( \left( D_1H_d(q_k,p_{k+1}) - p_k - \phi_d^q(q_k,p_{k+1}) \right) (\delta q_k)  \\
             &+  \left( D_2H_d(q_{k-1},p_k) - q_k - \phi_d^p(q_{k-1},p_k) \right) (\delta p_k) \big) \\
             &+ \left( D_1H_d(q_0,p_1) - \phi_d^q(q_0,p_1) \right) (\delta q_0) + \left( D_2H_d(q_{N-1},p_N) - \phi_d^p(q_{N-1},p_N) \right) (\delta p_N).
    \end{split}
  \end{equation}
	
  Using the same arguments as in the previous propositions, if
  $(q_\cdot,p_\cdot)$ is a trajectory and
  $(\delta q_\cdot,\delta p_\cdot)$ satisfies the corresponding
  boundary conditions, then the first expression
  in~\eqref{eq:discrete_Hamilton_equations-1} should vanish, which, in
  terms of the last expression proves
  equation~\eqref{eq:discrete_Hamilton_equations}.

  Conversely, if $(q_\cdot,p_\cdot)$
  satisfies~\eqref{eq:discrete_Hamilton_equations} and
  $(\delta q_\cdot,\delta p_\cdot)$ satisfies the corresponding
  boundary conditions, then the last expression
  in~\eqref{eq:discrete_Hamilton_equations-1} shows that it is a
  trajectory of the system.
\end{proof}

A similar approach to the variational principle used in
Definition~\ref{def:FDHS} has, also, been considered in
\cite{ar:deLeon_lainz_lopezGordon-discrete_hamilton_jacobi_theory_for_systems_with_external_forces}. A
difference between the two approaches is that we consider $q_\cdot$
and $p_\cdot$ to be independent, whereas in the cited work, they are
functionally related. This is reflected in the difference between our
equations~\eqref{eq:discrete_Hamilton_equations} and their (4a).

\begin{example}\label{ex:simple_with_kappa_and_nu-discrete-equations}
  The equations of motion~\eqref{eq:discrete_Hamilton_equations} for
  the of Example~\ref{ex:simple_with_kappa_and_nu-discrete} are
  \begin{equation*}
    \begin{cases}
      q_k = q_{k-1} + \frac{h}{m} p_k,\\
      p_k= p_{k+1} + h\nabla V(q_k)-h\, a(q_k,p_{k+1}),
    \end{cases}
    \stext{for} k=1,\ldots,N-1.
  \end{equation*}
  Interestingly, a relabeling of the first equations leads to
  \begin{equation*}
    \begin{cases}
      q_{k+1} = q_{k} + \frac{h}{m} p_{k+1}, \stext{ for } k=0,\ldots,N-2,\\
      p_{k+1}= p_{k} - h\nabla V(q_k) + h\, a(q_k,p_{k+1}), \stext{ for } k=1,\ldots,N-1,
    \end{cases}
  \end{equation*}
  that is the well known semi-explicit partitioned Euler
  method\footnote{This Euler method is sometimes called symplectic,
    which it is for the canonical symplectic structure in $T^*\R^n$
    \emph{when there are no forces}.}  of order $1$ applied
  to~\eqref{ex:simple_with_kappa_and_nu-equations} (see the expression
  (1.9) on p. 4
  of~\cite{bo:hairer_lubich_wanner-geometric_numerical_integration}). Moreover,
  this connection to the Euler method remains valid if the Hamiltonian
  and force used in Example~\ref{ex:simple_with_kappa_and_nu} are
  replaced by arbitrary $H$ and $\phi$.
\end{example}

\begin{example}\label{ex:trajectories_of_exact_FDHS_h=0}
  Let $\calM_d\CE (Q,H_d,\phi_d)$ be the forced discrete Hamiltonian
  system given by $H_d(q_0,p_1)\CE p_1q_0$ and $\phi_d(q_0,p_1)\CE
  0$.
  The equations of motion~\eqref{eq:discrete_Hamilton_equations}
  become
  \begin{equation*}
    q_k = q_{k-1},\quad p_k = p_{k+1} \stext{ for } k=1,\ldots,N-1.
  \end{equation*}
  Then, the trajectories of $\calM_d$ must satisfy
  \begin{equation*}
    q_0=\cdots=q_{N-1} \stext{ and } p_1=\cdots=p_N.
  \end{equation*}
\end{example}

Notice that the equations~\eqref{eq:discrete_Hamilton_equations} that
characterize the trajectories of an FDHS do not involve neither $p_0$
nor $q_N$. We could use a notion of trajectory in which those points
are absent, considering curves such as
$((q_0,p_1),\ldots,(q_{N-1},p_N))$, but we opt for keeping them and
asking the points to satisfy an additional condition, as it is done
in~\cite[Section
3.1]{ar:leok_zhang:2011:discrete_hamiltonian_variational_integrators}.

\begin{definition}
  Let $\calM_d \CE (Q,H_d,\phi_d)$ be a forced discrete Hamiltonian
  system. A trajectory
  $(q_\cdot,p_\cdot) \CE ((q_0,p_0),\ldots,(q_N,p_N))$ of $\calM_d$ is
  an \jdef{extended trajectory} of $\calM_d$ if it satisfies
  \begin{equation*}
    p_0 = D_1H_d (q_0,p_1) - \phi_d^q(q_0,p_1) \stext{and} 
    q_N = D_2H_d (q_{N-1},p_N) - \phi_d^p(q_{N-1},p_N).
  \end{equation*}
\end{definition}

\begin{remark}
  After~\eqref{eq:discrete_Hamilton_equations}, a discrete curve
  $(q_\cdot,p_\cdot) \CE ((q_0,p_0),\ldots,(q_N,p_N))$ is an extended
  trajectory of the FDHS $(Q,H_d,\phi_d)$ if and only if it satisfies
  \begin{equation}\label{eq:extended_discrete_Hamilton_equations}
    \begin{cases}
      q_{k+1} = D_2H_d (q_k,p_{k+1}) - \phi_d^p(q_k,p_{k+1}), \\
      p_k = D_1H_d (q_k,p_{k+1}) - \phi_d^q(q_k,p_{k+1}),
    \end{cases}
    \stext{for} k=0,\ldots,N-1.
  \end{equation}
\end{remark}

\begin{example}\label{ex:extended_trajectories_of_exact_FDHS_h=0}
  Continuing with Example~\ref{ex:trajectories_of_exact_FDHS_h=0}, we
  see that the extended trajectories of $\calM_d$ must satisfy
  \begin{equation*}
    q_0=\cdots=q_{N} \stext{ and } p_0=\cdots=p_N.
  \end{equation*}
  Thus, in terms of the discrete flow (to be discussed later),
  $\FlowD{\calM_d} = id_{T^*Q}$.
\end{example}

\begin{definition}
  Let $(Q,H_d,\phi_d)$ be an FDHS. We define its \jdef{forced discrete
    Legendre transforms} $\ff_{\phi_d}^\pm H_d : T^*Q \lra T^*Q$ by
  \begin{equation*}
    \begin{split}
      \ff_{\phi_d}^+ H_d (q_0,p_1) &\CE  \left( D_2H_d(q_0,p_1) -
                                     \phi_d^p(q_0,p_1) , p_1 \right), \\
      \ff_{\phi_d}^- H_d (q_0,p_1) &\CE  \left( q_0 , D_1H_d(q_0,p_1) -
                                     \phi_d^q(q_0,p_1) \right).
    \end{split}
  \end{equation*}
\end{definition}

\begin{remark}\label{rem:discrete_Hamilton_equations-ff}
  The forced discrete Hamilton
  equations~\eqref{eq:discrete_Hamilton_equations} of the FDHS
  $(Q,H_d,\phi_d)$ may be rewritten as
  \begin{equation}\label{eq:forced_discrete_Hamiton_equations-ff}
    \ff_{\phi_d}^+ H_d(q_{k-1},p_k) = \ff_{\phi_d}^- H_d(q_k,p_{k+1}),
    \stext{ for } k = 1,\ldots,N-1,
  \end{equation}
  and the condition for a trajectory to be extended is equivalent to
  \begin{equation}\label{eq:extended_forced_discrete_Hamiton_equations-ff}
    (q_0,p_0) = \ff_{\phi_d}^- H_d(q_0,p_1) \stext{and} 
    (q_N,p_N) = \ff_{\phi_d}^+ H_d(q_{N-1},p_N).
  \end{equation}
\end{remark}

\begin{definition}
  We say that an FDHS $\calM_d\CE (Q,H_d,\phi_d)$ is \jdef{regular} if
  the forced discrete Legendre transforms
  $\ff_{\phi_d}^\pm H_d : T^*Q \lra T^*Q$ are local diffeomorphisms;
  if they are diffeomorphisms, we say that $\calM_d$ is
  \jdef{hyperregular}.
\end{definition}

\begin{remark}\label{rem:condtions_for_discrete_regularity}
  Let $\calM_d\CE (Q,H_d,\phi_d)$ be a forced discrete Hamiltonian
  system. Then, using the triviality of $T^*Q\simeq Q\times Q^*$ and
  $TQ\simeq Q\times Q$, we see that the regularity condition for
  $\calM_d$ is equivalent to the nonsingularity (invertibility) of
  both $D_{12}H_d(q_0,p_1)-D_1\phi_d^p(q_0,p_1)$ and
  $ D_{21}H_d(q_0,p_1)-D_2\phi_d^q(q_0,p_1)$.
\end{remark}

For practical purposes, it is important to notice that the equations
of motion for trajectories and extended trajectories of FDHSs can be
decoupled (in $k$) and, so, can be solved iteratively instead of
having to solve them concurrently. We have the following result, whose
proof is a straightforward verification.

\begin{lemma}\label{le:discrete_extended_trajectories_N_and_1}
  Let $\calM_d\CE (Q,H_d,\phi_d)$ be a forced discrete Hamiltonian
  system and $(q_\cdot,p_\cdot) \CE ((q_0,p_0),\ldots,(q_N,p_N))$ be a
  discrete path in $T^*Q$. Then $(q_\cdot,p_\cdot)$ is an extended
  trajectory of $\calM_d$ if and only if each
  $((q_k,p_k),(q_{k+1},p_{k+1}))$ is an extended trajectory of
  $\calM_d$ for $k=0,\ldots,N-1$.
\end{lemma}

Algebraic equations do not always have solutions and the equations of
motion of an FDHS~\eqref{eq:discrete_Hamilton_equations}
(or~\eqref{eq:extended_discrete_Hamilton_equations}) are no exception,
even for regular systems. The next result establishes that, given a
(length-$1$) extended trajectory of an FDHS, a flow for the system can
be defined near that trajectory.

\begin{proposition}
  \label{prop:existence_of_discrete_flow_near_extended_trajectory}
  Let $\calM_d\CE (Q,H_d,\phi_d)$ be a regular forced discrete
  Hamiltonian system and $(q_\cdot,p_\cdot) \CE ((q_0,p_0),(q_1,p_1))$
  be an extended trajectory of $\calM_d$. Then, there are open subsets
  $U,V\subset T^*Q$ and a diffeomorphism $F^{\calM_d}:U\rightarrow V$
  such that
  \begin{enumerate}
  \item\label{it:existence_of_discrete_flow_near_extended_trajectory-domain}
    $(q_0,p_0)\in U$, $(q_1,p_1)\in V$ and
    $F^{\calM_d}(q_0,p_0) = (q_1,p_1)$.
  \item\label{it:existence_of_discrete_flow_near_extended_trajectory-flow}
    For any $(q_0',p_0')\in U$, if
    $(q_1',p_1') \CE F^{\calM_d}(q_0',p_0')$, then
    $((q_0',p_0'),(q_1',p_1'))$ is an extended trajectory of
    $\calM_d$.
  \item
    \label{it:existence_of_discrete_flow_near_extended_trajectory-trajectory}
    Any extended trajectory $((q_0',p_0'),(q_1',p_1'))$ of $\calM_d$
    such that $(q_0',p_0')\in U$ and $(q_1',p_1')\in V$ satisfies
    $(q_1',p_1') = F^{\calM_d}(q_0',p_0')$.
  \end{enumerate}
\end{proposition}

\begin{proof}
  Being $((q_0,p_0),(q_1,p_1))$ an extended trajectory of $\calM_d$,
  using~\eqref{eq:extended_forced_discrete_Hamiton_equations-ff} we
  have that $(q_0,p_0)=\ff_{\phi_d}^- H_d(q_0,p_1)$ and
  $(q_1,p_1)=\ff_{\phi_d}^+H_d(q_{0},p_1)$. As, by the
  regularity of $\calM_d$, both $\ff_{\phi_d}^\pm H_d$ are local
  diffeomorphisms, it is easy to see that there are open subsets
  $U,V,W\subset T^*Q$ such that $(q_0,p_0)\in U$, $(q_1,p_1)\in V$ and
  $(q_0,p_1)\in W$ and, also, $\ff_{\phi_d}^- H_d|_W^U$ and
  $\ff_{\phi_d}^+ H_d|_W^V$ are diffeomorphisms. Define
  \begin{equation}
    \label{eq:discrete_flow_in_terms_of_legendre_transforms}
    F^{\calM_d}:U\rightarrow V \stext{ by }
    F^{\calM_d} \CE  (\ff_{\phi_d}^+ H_d|_W^V)
    \circ (\ff_{\phi_d}^- H_d|_W^U)^{-1}.
  \end{equation}
  Then, as
  $F^{\calM_d}(q_0,p_0) = (\ff_{\phi_d}^+ H_d|_W^V) ((\ff_{\phi_d}^-
  H_d|_W^U)^{-1}(q_0,p_0)) = \ff_{\phi_d}^+ H_d(q_0,p_1) = (q_1,p_1)$,
  we see that
  point~\eqref{it:existence_of_discrete_flow_near_extended_trajectory-domain}
  is satisfied.

  For any $(q_0',p_0')\in U$, let
  \begin{equation*}
    (q_1',p_1') \CE  F^{\calM_d}(q_0',p_0') =
    (\ff_{\phi_d}^+ H_d|_W^V)((\ff_{\phi_d}^- H_d|_W^U)^{-1}(q_0',p_0')).
  \end{equation*}
  Then, by the explicit form of both $\ff_{\phi_d}^\pm H_d$ we see
  that $(\ff_{\phi_d}^- H_d|_W^U)^{-1}(q_0',p_0') =
  (q_0',p_1')$. Thus,
  \begin{gather*}
    (q_1',p_1') = \ff_{\phi_d}^+ H_d|_W^V(q_0',p_1') =
    \ff_{\phi_d}^+ H_d(q_0',p_1'),\\
    (q_0',p_0') = \ff_{\phi_d}^- H_d|_W^U(q_0',p_1') =
    \ff_{\phi_d}^- H_d(q_0',p_1'),
  \end{gather*}
  that, using~\eqref{eq:extended_forced_discrete_Hamiton_equations-ff},
  show that $((q_0',p_0'),(q_1',p_1'))$ is an extended trajectory of
  $\calM_d$. Hence,
  point~\eqref{it:existence_of_discrete_flow_near_extended_trajectory-flow}
  is valid.

  Last, if $((q_0',p_0'),(q_1',p_1'))$ is an extended trajectory of
  $\calM_d$ with $(q_0',p_0')\in U$ and $(q_1',p_1')\in V$,
  it
  satisfies~\eqref{eq:extended_forced_discrete_Hamiton_equations-ff},
  so that (using that the restriction of the Legendre transforms are
  diffeomorphisms)
  \begin{equation*}
    (q_1',p_1') = (\ff_{\phi_d}^+ H_d|_W^V)
    \circ (\ff_{\phi_d}^- H_d|_W^U)^{-1}(q_0',p_0')
    = F^{\calM_d}(q_0',p_0'),
  \end{equation*}
  proving the validity of
  point~\eqref{it:existence_of_discrete_flow_near_extended_trajectory-trajectory}.
\end{proof}

\begin{remark}
  \label{rem:discrete_flow_in_terms_of_forded_discrete_legendre_transforms}
  In the context of
  Proposition~\ref{prop:existence_of_discrete_flow_near_extended_trajectory},
  the regularity of $\calM_d$ ensures that both forced discrete
  Legendre transforms are locally invertible. The existence of a
  discrete trajectory ensures that the open sets where those
  transforms can be properly composed have nonempty intersection. In
  the proof of the Proposition we also obtain the explicit
  formula~\eqref{eq:discrete_flow_in_terms_of_legendre_transforms} for
  the discrete flow.
\end{remark}

We can depict the extended discrete trajectories as follows.
\begin{equation*}
  \xymatrixcolsep{1.5pc}\xymatrix{
    {} &
    {(q_0,p_1)} \ar@{|->}[dl]_{\ff_{\phi_d}^-H_d}
    \ar@{|->}[dr]^{\ff_{\phi_d}^+H_d}
    & {} & {(q_1,p_2)} \ar@{|->}[dl]_{\ff_{\phi_d}^-H_d} &
    {\cdots} &
    {(q_{N-1},p_N)} \ar@{|->}[dr]^{\ff_{\phi_d}^+H_d} &
    {}\\
    {(q_0,p_0)} \ar@{|->}[rr]_{F^{\calM_d}} & {} &
    {(q_1,p_1)} \ar@{|-}[r]
    & {} & {\cdots} & {} \ar@{->}[r] & {(p_N,q_N)}
  }
\end{equation*}


\section{Structural properties}
\label{sec:structural_properties}


\subsection{Relation with Lagrangian systems}
\label{sec:relation_with_lagrangian_systems}

It is well known that, under certain regularity conditions, there
exists a relation between the trajectories of (continuous) Lagrangian
and Hamiltonian systems. Specifically, given a Lagrangian system, the
Legendre transform can be used to construct a Hamiltonian system whose
trajectories are in correspondence with those of the original
Lagrangian system. We begin this section observing how this is
reflected in the discrete setting.

In this context, let us recall a few definitions. See Part 3
of~\cite{ar:marsden_west:2001:discrete_mechanics_and_variational_integrators}
and~\cite{ar:caruso_fernandez_tori_zuccalli:2023:lagrangian_reduction_of_forced_discrete_mechanical_systems}
for additional information.

\begin{definition}
  A \jdef{forced discrete Lagrangian system} (FDLS) is a triple
  $(Q,L_d,f_d)$ where $Q$ is a smooth manifold, the
  \jdef{configuration space}, $L_d : Q \times Q \lra \R$ is a smooth
  function, the \jdef{discrete Lagrangian} and
  $f_d \in \Omega^1(Q \times Q)$ is a differential $1$-form on
  $Q \times Q$, the \jdef{discrete force}.
\end{definition}

\begin{definition}
  A discrete curve $q_\cdot\CE (q_0,\ldots,q_N)$ is a
  \jdef{trajectory} of the FDLS $(Q,L_d,f_d)$ if it satisfies
  \begin{equation}\label{eq:forced-variational principle}
    \delta \left( \sum_{k=0}^{N-1} L_d(q_{k},q_{k+1})  \right) +
    \sum_{k=0}^{N-1} f_d(q_k,q_{k+1})(\delta q_k,\delta q_{k+1}) = 0,
  \end{equation}
  for all infinitesimal variations $\delta q_\cdot$ over $q_\cdot$
  with fixed endpoints, that is, $\delta q_0=0$ and $\delta q_N=0$.
\end{definition}

\begin{theorem}\label{th:forced_EL_equations}
  Let $(Q,L_d,f_d)$ be a FDLS. Then, a discrete curve
  $q_\cdot : \{ 0,\ldots,N \} \lra Q$ is a trajectory of $(Q,L_d,f_d)$
  if and only if the following algebraic identities are satisfied:
  \begin{equation}\label{eq:forced_EL_equations}
    D_2 L_d(q_{k-1},q_k) + D_1 L_d(q_k,q_{k+1}) + f_d^+(q_{k-1},q_k) +
    f_d^-(q_k,q_{k+1}) = 0 \in T_{q_k}^{*}Q
  \end{equation}
  for all $k = 1,\ldots,N-1$, known as \jdef{forced discrete
    Euler--Lagrange equations}.
\end{theorem}

Notice that in~\eqref{eq:forced_EL_equations} we are taking advantage
of the Cartesian product structure of $Q\times Q$ to decompose
$f_d = f_d^- + f_d^+$.

\begin{definition}\label{def:forced_discrete_lergende_transforms-lagrangian}
  Given a FDLS $\calM_d\CE (Q,L_d,f_d)$, the \jdef{forced discrete
    Legendre transforms} are the maps
  $\ff^+_{f_d}L_{d}: Q \times Q \lra T^*Q$ and
  $\ff^-_{f_d}L_{d} : Q \times Q \lra T^*Q$ defined by
  \begin{equation*}
    \begin{split}
      \ff^+_{f_d}L_{d}(q_0,q_1) &\CE  D_2 L_d(q_0,q_1) +
                                  f_d^+(q_0,q_1) \in T_{q_1}^*Q, \\
      \ff^-_{f_d}L_{d}(q_0,q_1) &\CE  -D_1 L_d(q_0,q_1) -
                                  f_d^-(q_0,q_1) \in T_{q_0}^*Q.
    \end{split}
  \end{equation*}
  When $\ff^+_{f_d}L_{d}$ and $\ff^-_{f_d}L_{d}$ are local
  diffeomorphisms, $\calM_d$ is said to be \jdef{regular} and, if they
  are (global) diffeomorphisms, $\calM_d$ is said to be \jdef{hyperregular}.
\end{definition}

Notice that the forced discrete Euler--Lagrange
equations~\eqref{eq:forced_EL_equations} may be written as
\begin{equation}\label{eq:forced_EL_equations-ff}
  \ff^+_{f_d}L_{d}(q_{k-1},q_{k}) = \ff^-_{f_d}L_{d}(q_{k},q_{k+1})
  \stext{ for } k=1,...,N-1.
\end{equation}

In what follows, we will use the second components of both transforms,
so it will be useful to define the functions
\begin{equation*}
  \begin{split}
    \conj{\ff^+_{f_d}L_{d}}(q_0,q_1) &\CE  (\pr_2 \circ \ff^+_{f_d}L_{d})(q_0,q_1) \in Q^* \\
    \conj{\ff^-_{f_d}L_{d}}(q_0,q_1) &\CE  (\pr_2 \circ \ff^-_{f_d}L_{d})(q_0,q_1) \in Q^*,
  \end{split}
\end{equation*}
where $\pr_2 : Q \times Q^* \lra Q^*$ is the projection onto the second factor.

Let us consider the function
$\widetilde{\ff}_{f_d}L_d : Q \times Q \lra Q \times Q^*$ given by
\begin{equation*}
  \widetilde{\ff}_{f_d} L_d(q,q') \CE
  \left(q,\conj{\ff_{f_d}^+ L_d}(q,q')\right).
\end{equation*}

\begin{definition}
  We say that a FDLS $(Q,L_d,f_d)$ satisfies the \jdef{modified
    regularity condition} (MRC) if $\widetilde{\ff}_{f_d} L_d$ is a
  local diffeomorphism. If $\widetilde{\ff}_{f_d} L_d$ is a (global)
  diffeomorphism, then we say that it satisfies the \jdef{modified
    hyperregularity condition} (MHC).
\end{definition}

It is easy to find examples that show that the regularity of a FDLS
does not guarantee the satisfaction of the MRC.

Let $(Q,L_d,f_d)$ be a FDLS that satisfies the MHC. We define
$H_d : T^*Q \lra \R$ by
\begin{equation}\label{eq:H_d_from_L_d}
  H_d(q,p) \CE  p q^+ - L_d(q,q^+), \stext{where $q^+ \in Q$ satisfies}
  p = \conj{\ff_{f_d}^+ L_d}(q,q^+).
\end{equation}

Since $\widetilde{\ff}_{f_d} L_d$ is a diffeomorphism, $q^+ \CE  (\pr_2 \circ (\widetilde{\ff}_{f_d} L_d)^{-1})(q,p)$, so that $H_d$ is well defined. In other words, we have a function $q^+ : Q \times Q^* \lra Q$ that satisfies $(\widetilde{\ff}_{f_d} L_d)^{-1} (q,p) = (q,q^+(q,p))$.

The other ingredient required to construct an FDHS associated to
$(Q,L_d,f_d)$ is a $1$-form on $T^*Q$. With this in mind, we consider
\begin{equation}\label{eq:phi_d_from_f_d}
  \phi_d \CE  ((\widetilde{\ff}_{f_d} L_d)^{-1})^* f_d.    
\end{equation}

Unraveling the definitions, we obtain
\begin{equation}\label{eq:phi_d_from_f_d-explicit}
  \begin{split}
    \phi_d^q(q,p)(\delta q) &= \left( f_d^-(q,q^+(q,p)) +
      f_d^+(q,q^+(q,p)) D_1 q^+ (q,p) \right) (\delta q) \\
    \phi_d^p(q,p)(\delta p) &= f_d^+(q,q^+(q,p)) D_2q^+(q,p)(\delta p).
  \end{split}
\end{equation}

In summary, we have constructed an FDHS $(Q,H_d,\phi_d)$ from a FDLS
$(Q,L_d,f_d)$ that satisfies the MHC.

Next, we study the relation between the trajectories of both
systems. Given a discrete curve $q_\cdot \CE (q_0,\ldots,q_N)$ in $Q$,
we define a discrete curve in $T^*Q$ by
\begin{equation}\label{eq:curve_in_T^*Q_from_curve_in_Q}
  (q_k,p_k) \CE
  \begin{cases}
    \ff_{f_d}^+ L_d (q_{k-1},q_k), \stext{ if } k = 1,\ldots,N,\\    
    \ff_{\phi_d}^-H_d(q_0,p_1), \stext{ if } k=0.
  \end{cases}
\end{equation}

\begin{lemma}\label{le:curves_in_Q_and_in_T^*Q}
  Let $(Q,L_d,f_d)$ be a FDLS that satisfies the MHC and let
  $(Q,H_d,\phi_d)$ be the FDHS constructed by~\eqref{eq:H_d_from_L_d}
  and~\eqref{eq:phi_d_from_f_d}. Then,
  \begin{enumerate}
  \item\label{it:curves_in_Q_and_in_T^*Q-1} for every discrete curve
    $q_\cdot = (q_0,\ldots,q_N)$ in $Q$, the discrete curve
    $(q_\cdot,p_\cdot) = ((q_0,p_0),\ldots,(q_N,p_N))$ in $T^*Q$
    constructed via~\eqref{eq:curve_in_T^*Q_from_curve_in_Q} satisfies
    the first equation
    in~\eqref{eq:extended_discrete_Hamilton_equations} as well as the
    $k=0$ case of the second.

  \item\label{it:curves_in_Q_and_in_T^*Q-2} Conversely, let
    $(q_\cdot,p_\cdot) \CE ((q_0,p_0),\ldots,(q_N,p_N))$ be a discrete
    curve in $T^*Q$ that satisfies the first equation
    in~\eqref{eq:extended_discrete_Hamilton_equations} as well as the
    $k=0$ case of the second. Then, $(q_\cdot,p_\cdot)$ is constructed
    from $q_\cdot \CE (q_0,\ldots,q_N)$
    by~\eqref{eq:curve_in_T^*Q_from_curve_in_Q}.
  \end{enumerate}
\end{lemma}

\begin{proof}
  Computing, using~\eqref{eq:H_d_from_L_d}
  and~\eqref{eq:phi_d_from_f_d-explicit}, we have
  \begin{equation}\label{eq:D2Hd}
    \begin{split}
      D_2H_d(q,p) &= D_2 (p q^+(q,p) - L_d(q,q^+(q,p))) \\
                  &= q^+(q,p) + \phi_d^p(q,p).
    \end{split}
  \end{equation}

  \begin{enumerate}
  \item For every $k = 0,\ldots,N-1$, taking $q \CE q_{k}$ and
    $p \CE p_{k+1} \CE \conj{\ff_{f_d}^+ L_d}(q_{k},q_{k+1})$, we
    have $q^+(q,p) = q^+(q_{k},p_{k+1}) = q_{k+1}$. Then,
    \begin{equation*}
      D_2H_d(q_{k},p_{k+1}) \stackrel{\eqref{eq:D2Hd}}{=}
      q^+(q_k,p_{k+1}) + \phi_d^p(q_k,p_{k+1}) = q_{k+1} + \phi_d^p(q_k,p_{k+1}),
    \end{equation*}
    which is equivalent to the first equation
    in~\eqref{eq:extended_discrete_Hamilton_equations}.  For $k=0$,
    comparison of~\eqref{eq:curve_in_T^*Q_from_curve_in_Q} with the
    second equation of~\eqref{eq:extended_discrete_Hamilton_equations}
    shows that the conditions are equivalent.

  \item For every $k = 0,\ldots,N-1$, taking $q \CE q_k$ and
    $p \CE p_{k+1}$, the first equation
    in~\eqref{eq:extended_discrete_Hamilton_equations} says that
    \begin{equation*}
      q^+(q_k,p_{k+1}) \stackrel{\eqref{eq:D2Hd}}{=}
      D_2H_d(q_k,p_{k+1}) - \phi_d^p(q_k,p_{k+1})
      \stackrel{\eqref{eq:extended_discrete_Hamilton_equations}}{=} q_{k+1}.
    \end{equation*}

    Hence, by the definition of $q^+(q_k,p_{k+1})$,
    \begin{equation*}
      p_{k+1} = \conj{\ff_{f_d}^+ L_d}(q_k,q^+(q_k,p_{k+1})) =
      \conj{\ff_{f_d}^+ L_d}(q_k,q_{k+1}),
    \end{equation*}
    so that
    \begin{equation*}
      \ff_{f_d}^+ L_d (q_k,q_{k+1}) =
      \left( q_{k+1} , \conj{\ff_{f_d}^+ L_d}(q_k,q_{k+1})\right) =
      (q_{k+1}, p_{k+1}),
    \end{equation*}
    proving that the case $k=1,\ldots,N$
    of~\eqref{eq:curve_in_T^*Q_from_curve_in_Q} holds.  The case $k=0$
    of~\eqref{eq:curve_in_T^*Q_from_curve_in_Q} is equivalent to the
    $k=0$ case of the second equation
    of~\eqref{eq:extended_discrete_Hamilton_equations}, completing the
    proof.
  \end{enumerate}
\end{proof}

\begin{proposition}
  Let $\calM_d\CE (Q,L_d,f_d)$ be a FDLS that satisfies MHC.  Given
  discrete curves $q_\cdot\CE (q_0,\ldots,q_N)$ and
  $(q_\cdot,p_\cdot)\CE ((q_0,p_0),\ldots,(q_N,p_N))$ related as
  in~\eqref{eq:curve_in_T^*Q_from_curve_in_Q}, $q_\cdot$ is a
  trajectory of $\calM_d$ if and only if $(q_\cdot,p_\cdot)$ is an
  extended trajectory of the FDHS $(Q,H_d,\phi_d)$ constructed
  using~\eqref{eq:H_d_from_L_d} and~\eqref{eq:phi_d_from_f_d} from
  $\calM_d$.
\end{proposition}

\begin{proof}
  Since $q_\cdot$ is a discrete curve in $Q$ and $(q_\cdot,p_\cdot)$
  is a discrete curve in $T^*Q$ that are related
  by~\eqref{eq:curve_in_T^*Q_from_curve_in_Q}, due to
  Lemma~\ref{le:curves_in_Q_and_in_T^*Q}, we only have to prove that
  $q_\cdot$ satisfies~\eqref{eq:forced_EL_equations-ff} if and only if
  $(q_\cdot,p_\cdot)$ satisfies the second equation
  in~\eqref{eq:extended_discrete_Hamilton_equations}, for
  $k=1,\ldots,N-1$.
  
  Using~\eqref{eq:H_d_from_L_d} and,
  then,~\eqref{eq:phi_d_from_f_d-explicit}, for arbitrary $(q,p)$, we
  obtain
  \begin{equation*}
    D_1H_d(q,p) = \phi_d^q(q,p) + \conj{\ff_{f_d}^- L_d}(q,q^+(q,p)).
  \end{equation*}
	
  For any $q_\cdot$ and $(q_\cdot,p_\cdot)$ related
  by~\eqref{eq:curve_in_T^*Q_from_curve_in_Q}, we evaluate the
  previous identity at $q \CE q_k$ and
  $p \CE p_{k+1} \CE \conj{\ff_{f_d}^+ L_d}(q_k,q_{k+1})$ with
  $k=0,\ldots,N-1$:
  \begin{equation}\label{eq:D1Hd en qk-1 pk}
    \begin{split}
      D_1H_d(q_k,p_{k+1}) &= \phi_d^q(q_k,p_{k+1}) +
                            \conj{\ff_{f_d}^- L_d}(q_k,q^+(q_k,p_{k+1})) \\
                          &= \phi_d^q(q_k,p_{k+1}) +
                            \conj{\ff_{f_d}^- L_d}(q_k,q_{k+1}).
    \end{split}
  \end{equation}
	
  If $q_\cdot$ is a trajectory of $(Q,L_d,f_d)$, then it
  satisfies~\eqref{eq:forced_EL_equations-ff}, which, used
  in~\eqref{eq:D1Hd en qk-1 pk} for $k=1,\ldots,N-1$ and
  taking~\eqref{eq:curve_in_T^*Q_from_curve_in_Q} into account, says
  that $(q_\cdot,p_\cdot)$ satisfies the second equation
  in~\eqref{eq:extended_discrete_Hamilton_equations} for $k$ in that
  range.

  Conversely, if $(q_\cdot,p_\cdot)$ is an extended trajectory of
  $(Q,H_d,\phi_d)$, then it satisfies the second equation
  in~\eqref{eq:extended_discrete_Hamilton_equations},
  and~\eqref{eq:D1Hd en qk-1 pk} says that
  \begin{equation*}
    p_k = \conj{\ff_{f_d}^- L_d}(q_k,q_{k+1}).
  \end{equation*}
  Using now the case $k=1,\ldots,N-1$
  of~\eqref{eq:curve_in_T^*Q_from_curve_in_Q} we see
  that~\eqref{eq:forced_EL_equations-ff} is satisfied and $q_\cdot$ is
  a trajectory of $(Q,L_d,f_d)$.
\end{proof}


\subsection{The canonical symplectic structure}
\label{sec:the_canonical_symplectic_structure}

We now study the evolution of the canonical symplectic structure of
$T^*Q$ by the flow of an FDHS. In
\cite{ar:leok_zhang:2011:discrete_hamiltonian_variational_integrators},
the authors study this for unforced systems, where one expects this
structure to be conserved.

Given an FDHS $(Q,H_d,\phi_d)$, let us define the $2$-forms
$\omega_{\phi_d}^\pm \CE (\ff_{\phi_d}^\pm H_d)^* \omega_Q$, where
$\omega_Q$ is the canonical $2$-form on $T^*Q$. In coordinates, using
the decomposition~\eqref{eq:phi_d_qp_decomposition},
\begin{equation*}
  \phi_d = \phi_d^q  + \phi_d^p  = \phi_{d,j}^q \ dq^j +
  \phi_d^{p,j} \ dp_j
\end{equation*}
we have
\begin{equation*}
  (\ff_{\phi_d}^+ H_d)^*(q^i \ dp_i) =
  \left( \frac{\partial H_d}{\partial p_i} - \phi_d^{p,i} \right) \ dp_i
\end{equation*}
and, therefore,
\begin{equation*}
  \begin{split}
    (\ff_{\phi_d}^+ H_d)^*(\omega_Q) &= \frac{\partial^2 H_d}{\partial q^j \partial p_i} \ dq^j \wedge dp_i + \frac{\partial^2 H_d}{\partial p_j \partial p_i} \ dp_j \wedge dp_i - d\phi_d^{p,i} \wedge dp_i \\
                                     &= \frac{\partial^2 H_d}{\partial q^j \partial p_i} \ dq^j \wedge dp_i - d\phi_d^{p,i} \wedge dp_i.
  \end{split}
\end{equation*}

A similar computation shows that
\begin{equation*}
  (\ff_{\phi_d}^- H_d)^*(\omega_Q) =
  \frac{\partial^2 H_d}{\partial q^j \partial p_i} \ dq^j \wedge dp_i +
  d\phi_{d,j}^q \wedge dq^j.
\end{equation*}

Defining
\begin{equation*}
  \omega_{H_d} \CE
  \frac{\partial^2 H_d}{\partial q^j \partial p_i} \ dq^j \wedge dp_i,
\end{equation*}
we have proved the following result:

\begin{proposition}
  If $(Q,H_d,\phi_d)$ is an FDHS, then
  \begin{equation*}
    \omega_{\phi_d}^+ = \omega_{H_d} - d\phi_d^p \wedge dp \stext{and}
    \omega_{\phi_d}^- = \omega_{H_d} + d\phi_d^q \wedge dq.
  \end{equation*}
\end{proposition}

\begin{corollary}\label{cor:omega_difference}
  If $(Q,H_d,\phi_d)$ is an FDHS, then
  \begin{equation*}
    \omega_{\phi_d}^+ - \omega_{\phi_d}^- = -d\phi_d.
  \end{equation*}
\end{corollary}

Let $\calM_d\CE (Q,H_d,\phi_d)$ be a regular FDHS. We are interested
in the evolution of $\omega_Q$ by the discrete flow
$(q_k,p_k)\mapsto (q_{k+1},p_{k+1})$ of the system. If we denote this
flow by $\FlowD{\calM_d}$,
recalling~\eqref{eq:forced_discrete_Hamiton_equations-ff}
and~\eqref{eq:extended_forced_discrete_Hamiton_equations-ff}, we can
express
$\FlowD{\calM_d} = \left( \ff_{\phi_d}^+ H_d \right) \circ \left(
  \ff_{\phi_d}^- H_d \right)^{-1}$ and, using the previous
computations, derive the following result.

\begin{proposition}\label{prop:evolution_of_omega_Q_with_forces}
  Let $(Q,H_d,\phi_d)$ be a regular FDHS. Then, the evolution of the
  canonical symplectic form $\omega_Q$ is given by
  \begin{equation*}\label{eq:evolution_of_omega_Q_with_forces}
    {\FlowD{\calM_d}}^* (\omega_Q) = \omega_Q -
    \left( \left( \ff_{\phi_d}^- H_d \right)^{-1} \right)^*(d\phi_d).
  \end{equation*}
\end{proposition}

\begin{remark}
  Notice that if $\phi_d$ is closed (in particular, if
  $\phi_d \equiv 0$), we have that $\omega_Q$ is preserved by the flow
  of the system.
\end{remark}

\begin{example}
  \label{ex:damped_harmonic_oscillator-evolution_of_symplectic_structure}
  A discretization of a damped harmonic oscillator is given by the
  FDHS $(Q,H_d,\phi_d)$ where $Q\CE \R$,
  \begin{equation*}
    H_d(q,p) \CE  pq + h \frac{p^2}{2m} + \frac{h}{2} \nu q^2
    \stext{ and }
    \phi_d(q,p) \CE  -h \kappa \frac{p}{m} \ dq,
  \end{equation*}
  where $m$, $\nu$ and $\kappa$ are constants and $h\neq 0$ is a fixed
  time-step.

  The forced discrete Legendre transforms are
  \begin{equation*}
    (\ff_{\phi_d}^+ H_d)(q,p) = \left( q + h \frac{p}{m} , p \right),
    \stext{ and }
    (\ff_{\phi_d}^- H_d)(q,p) = \left( q ,
      p + h \nu q + h \kappa \frac{p}{m} \right).
  \end{equation*}

  If $1 + \frac{h\kappa}{m} \neq 0$, we have
  \begin{equation*}
    (\ff_{\phi_d}^- H_d)^{-1} (q',p') =
    \left( q' , \frac{p' - h \nu q'}{1 + \frac{h\kappa}{m}} \right) =
    \left( q' , \frac{m}{m + h\kappa} (p' - h \nu q') \right).
  \end{equation*}

  We see that
  \begin{equation*}
    \begin{split}
      \left( \left( \ff_{\phi_d}^- H_d \right)^{-1} \right)^* (d\phi_d) &=\frac{h\kappa}{m} \ dq' \wedge \left( \frac{m}{m + h\kappa} (dp' - h \nu dq') \right) \\
                                                                        &= \frac{h\kappa}{m + h\kappa} (dq' \wedge dp') = \frac{h\kappa}{m + h\kappa} \omega_Q,
    \end{split}
  \end{equation*}
  and, then, according to
  Proposition~\ref{prop:evolution_of_omega_Q_with_forces},
  \begin{equation*} {\FlowD{\calM_d}}^* (\omega_Q) = \omega_Q -
    \frac{h\kappa}{m + h\kappa} \omega_Q = \left( 1 - \frac{h\kappa}{m +
        h\kappa} \right) \ \omega_Q.
  \end{equation*}

  Notice that if $\kappa = 0$ (i.e., the force vanishes), the
  canonical symplectic structure is preserved by the flow.
\end{example}


\subsection{The canonical momentum map}
\label{sec:the_canonical_momentum_map}

Let $\calM_d\CE (Q,H_d,\phi_d)$ be an FDHS and let $G$ be a Lie group
acting on the left on $Q$, by a free and proper action $l^Q$. Let
$l^{Q^*}$ be the induced action on $Q^*$. Following the ideas for the
unforced case found in \cite[Section
5]{ar:leok_zhang:2011:discrete_hamiltonian_variational_integrators},
$G$ is said to be a \jdef{symmetry group} for the system
$(Q,H_d,\phi_d)$ if the function $R_d : Q \times Q \times Q^* \lra \R$
given by
\begin{equation*}
  R_d(q_0,q_1,p_1) \CE  p_1 q_1 - H_d(q_0,p_1)
\end{equation*}
is $G$-invariant for the action
$l_g^{Q \times Q \times Q^*}(q_0,q_1,p_1) \CE
(l_g^Q(q_0),l_g^Q(q_1),l_g^{Q^*}(p_1))$.

Let $J : T^*Q \lra \frakg^*$ be the canonical momentum map, defined by
$J(q,p)\xi \CE p(\xi_Q(q))$ for any $\xi \in \frakg$. Given an
extended trajectory $((q_0,p_0),(q_1,p_1))$ of $\calM_d$ and
$\varepsilon > 0$, let
$q_i^\varepsilon \CE l^Q_{\exp(\varepsilon \xi)}(q_i)$ and
$p_i^\varepsilon \CE l^{Q^*}_{\exp(\varepsilon \xi)}(p_i)$, with
$i = 0,1$. Since $R_d$ is $G$-invariant, a computation similar to that
of~\cite{ar:leok_zhang:2011:discrete_hamiltonian_variational_integrators}
yields
\begin{equation*}
  \begin{split}
    0 & = \frac{d}{d\varepsilon} \left( p_1^\varepsilon q_1^\varepsilon - H_d(q_0^\varepsilon,p_1^\varepsilon) \right) \bigg|_{\varepsilon=0} \\	
      & = p_1 \left. \frac{d}{d\varepsilon} q_1^\varepsilon \right|_{\varepsilon=0} +  q_1 \left. \frac{d}{d\varepsilon} p_1^\varepsilon \right|_{\varepsilon=0} - D_1H_d(q_0,p_1) \left. \frac{d}{d\varepsilon} q_0^\varepsilon \right|_{\varepsilon=0} - D_2H_d(q_0,p_1) \left. \frac{d}{d\varepsilon} p_1^\varepsilon \right|_{\varepsilon=0} \\
      & = p_1 \left. \frac{d}{d\varepsilon} q_1^\varepsilon \right|_{\varepsilon=0} +  q_1 \left. \frac{d}{d\varepsilon} p_1^\varepsilon \right|_{\varepsilon=0} - \left(p_0 + \phi^q_d(q_0,p_1) \right) \left. \frac{d}{d\varepsilon} q_0^\varepsilon \right|_{\varepsilon=0} - \left(q_1 + \phi^p_d(q_0,p_1) \right) \left. \frac{d}{d\varepsilon} p_1^\varepsilon \right|_{\varepsilon=0} \\
      & = p_1 \left. \frac{d}{d\varepsilon} q_1^\varepsilon \right|_{\varepsilon=0} - p_0 \left. \frac{d}{d\varepsilon} q_0^\varepsilon \right|_{\varepsilon=0} - \phi^q_d(q_0,p_1)  \left. \frac{d}{d\varepsilon} q_0^\varepsilon \right|_{\varepsilon=0} -  \phi^p_d(q_0,p_1) \left. \frac{d}{d\varepsilon} p_1^\varepsilon \right|_{\varepsilon=0} \\
      & = p_1 \xi_Q(q_1) - p_0 \xi_Q(q_0) - \phi_d(q_0,p_1)  \xi_{Q \times Q^*}(q_0,p_1).
\end{split}
\end{equation*}

Rewriting the previous identity in terms of the canonical momentum map
$J$, we have the following result.

\begin{proposition}\label{prop:momentum_map_evolution}
  Let $\calM_d\CE (Q,H_d,\phi_d)$ be an FDHS and let $G$ be a symmetry
  group of the system. Then, the canonical momentum map
  $J : T^*Q \lra \R$ evolves according to
  \begin{equation*}
    J(q_1,p_1) \xi = J(q_0,p_0) \xi +
    \phi_d(q_0,p_1)  \left( \xi_{Q \times Q^*}(q_0,p_1) \right),
  \end{equation*}
  where $((q_0,p_0),(q_1,p_1))$ is an extended trajectory of $\calM_d$.

  In particular, if
  $\phi_d(q_0,p_{1})(\xi_{Q \times Q^*}(q_0,p_{1})) = 0$, the
  canonical momentum map is preserved along the extended trajectories
  of the system.
\end{proposition}

\begin{example}\label{example:unit_mass_particle_with_radial_potential}
  Consider a unit mass particle moving in the plane with radial
  potential and friction-type forcing (as seen in~\cite[Example
  2.3]{ar:caruso_fernandez_tori_zuccalli:2023:lagrangian_reduction_of_forced_discrete_mechanical_systems}
  and, originally, in~\cite[Example
  3.2.3]{ar:marsden_west:2001:discrete_mechanics_and_variational_integrators})
  from the Hamiltonian point of view. This leads to the forced
  Hamiltonian system $(Q,H,\phi)$ that, in polar coordinates, is given
  by $Q\CE \R^+ \times S^1$,
  \begin{equation*}
    \begin{split}
      H(r,\theta,p^r,p^\theta) \CE& \frac{1}{2} \left( (p^r)^2 + \frac{(p^\theta)^2}{r^2} \right) + r^2 (r^2 - 1)^2,\\
      \phi(r,\theta,p^r,p^\theta) \CE& -\mu \left( p^r dr + p^\theta d\theta \right),
    \end{split}
  \end{equation*}
  where $\mu$ is the constant coefficient of friction. For convenience
  of computation, we consider, instead, its lift to its covering
  space, $Q\CE \R^+\times\R$.

  A possible discretization is the FDHS $(Q,H_d,\phi_d)$ given by
  \begin{equation*}
    \begin{split}
      H_d(r,\eta,p^r,p^\eta) &= p^r r + p^\eta \eta + \frac{h}{2} \left( (p^r)^2 + \frac{(p^\eta)^2}{r^2} \right) + h r^2 (r^2 - 1)^2, \\
      \phi_d(r,\eta,p^r,p^\eta) &= - h \mu \left[ p^r dr + p^\eta d\eta \right],
    \end{split}
  \end{equation*}
  where $h\neq 0$ is a constant.

  The rotational invariance of the continuous system becomes a
  translational invariance in the $\eta$ variable for the lifted
  system. The Lie group $\R$ is a symmetry group of the system, since
  \begin{equation*}
    \begin{split}
      R_d(q_0,q_1,p_1) &\CE  \langle p_1,q_1 \rangle - H_d(q_0,p_1) \\
                       &= p^r_1 (r_1 - r_0) + p^\eta_1 (\eta_1 - \eta_0) -
                         \frac{h}{2} \left( (p^r_1)^2 +
                         \frac{(p^\eta_1)^2}{r_0^2} \right) -
                         h r_0^2 (r_0^2 - 1)^2
    \end{split}
  \end{equation*}
  is invariant by the action
  \begin{equation*}
    l_g^{Q \times Q \times Q^*}(r_0,\eta_0,r_1,\eta_1,p^r_1,p^\eta_1) \CE
    (r_0, \eta_0 + g, r_1, \eta_1 + g, p^r_1, p^\eta_1), \quad g \in \R.
  \end{equation*}

  For $\xi \in \R=\operatorname{Lie}(\R)$, the infinitesimal generator
  for the action is
  \begin{equation*}
    \xi_{Q \times Q^*}(q_0,p_1) =
    \dert l^{Q \times Q^*}_{t\xi}(r_0,\eta_0,p^r_1,p^\eta_1) = (0,\xi,0,0).
  \end{equation*}
  Therefore,
  \begin{equation*}
    \phi_d(q_0,p_1) \left( \xi_{Q \times Q^*}(q_0,p_1) \right) =
    - h \mu \left( p^r_1 \cdot 0 + p^\eta_1 \xi \right) = - h \mu p^\eta_1 \xi
  \end{equation*}
  and Proposition \ref{prop:momentum_map_evolution} yields
  \begin{equation*}
    J(q_1,p_1) = J(q_0,p_0) - h \mu p^\eta_1.    
  \end{equation*}
\end{example}


\section{Error analysis: discrete exact systems}
\label{sec:error_analysis-discrete_exact_systems}

So far, we have considered forced discrete Hamiltonian systems as
discrete-time dynamical systems, mostly independent of their, more
common, continuous-time counterparts. In this section we want to see
how, given a (continuous) forced Hamiltonian system $\calM$, there is
a one parameter family of FDHSs $\calM_{d,h}^E$ ---known as the
\jdef{discrete exact} systems---, whose trajectories (for a given $h$)
interpolate the trajectories of $\calM$ at discrete time steps
(usually $h$).

\begin{example}\label{ex:exact_discrete_hamiltonian_system}
  Let $\calM\CE (Q,H,\phi)$ be a forced Hamiltonian system. For any $h\neq 0$
  we define $H_{d,h}^E:T^*Q\rightarrow \R$ by
  \begin{equation}\label{eq:exact_discrete_hamiltonian_system-II-hamiltonian}
    H_{d,h}^E(q_0,p_1) \CE  p(h) q(h) -
    \int_0^h (p(t) \dot{q}(t) -H(q(t),p(t))) dt,
  \end{equation}
  where $(q(t),p(t)) \CE \FlowCBV{\calM}{h}{t}(q_0,p_1)$ is the
  trajectory of $\calM$ that satisfies the boundary conditions
  $q(0)=q_0$ and $p(h)=p_1$. Similarly, we define
  $\phi_{d,h}^E\in \Omega^1(T^*Q)$ by
  \begin{equation}\label{eq:exact_discrete_hamiltonian_system-II-force}
    \phi_{d,h}^E(q_0,p_1)(\delta q_0,\delta p_1) \CE 
    \int_0^h \phi(q(t),p(t))T_{(q_0,p_1)}
    \FlowCBV{\calM}{h}{t}(\delta q_0,\delta p_1) dt.
  \end{equation}
  Then, $\calM_{d,h}^E\CE (Q,H_{d,h}^E,\phi_{d,h}^E)$ is a discrete
  Hamiltonian system that we call the \jdef{exact forced discrete
    Hamiltonian system} associated to $\calM$. This notion can be
  extended to the case $h=0$ with
  \begin{equation}\label{eq:exact_discrete_hamiltonian_system-II-h=0}
    H_{d,0}^E(q_0,p_1) \CE  p_1q_0 \stext{ and }
    \phi_{d,0}^E(q_0,p_1) \CE 0.
  \end{equation}
\end{example}

It is not clear that the family of systems $\calM_{d,h}^E$ described
in Example~\ref{ex:exact_discrete_hamiltonian_system} is well defined
because it relies on the existence of the flow associated to boundary
value
problems. Proposition~\ref{prop:existence_of_exact_discrete_hamiltonian_system-local}
shows that, under certain conditions, $\calM_{d,h}^E$ is well defined.

Let $\calM\CE (Q,H,\phi)$ be a forced Hamiltonian system and
$(q_0,p_1)\in T^*Q = Q\times Q^*$. Then, by
Proposition~\ref{prop:flow_of_hamilton_equations-b}, there exist
constants $r,r'>0$ and $T_-<T_+$ such that $0\in (T_-,T_+)$ and, for
any $T\in [T_-,T_+]$ and
$(q_0',p_1') \in \conj{B^\infty_{r'}(q_0,p_1)} \subset Q\times
Q^*$, the boundary value
problem~\eqref{eq:flow_of_hamilton_equations-b-II-eqs} has a unique
solution $(q(t),p(t))$ that satisfies
$(q(0),p(0))\in B^\infty_r(q_0,p_1)$. Fix $T\in [T_-,T_+]$; given
$(q_0',p_1')\in B^\infty_{r'}(q_0,p_1)$, let $(q_T(t),p_T(t))$ be the
unique solution of~\eqref{eq:flow_of_hamilton_equations-b-II-eqs} such
that $(q_T(0),p_T(0))\in B^\infty_r(q_0,p_1)$. With this information
we define
\begin{equation}
  \label{eq:exact_discrete_hamiltonian_system-II-hamiltonian-local}
  H_{d,T}^E(q_0',p_1') \CE  p_T(T) q_T(T) -
  \int_0^T (p_T(t) \dot{q_T}(t) -H(q_T(t),p_T(t))) dt,
\end{equation}
and
\begin{equation}\label{eq:exact_discrete_hamiltonian_system-II-force-local}
  \phi_{d,T}^E(q_0',p_1')(\delta q_0',\delta p_1') \CE 
  \int_0^T \phi(q_T(t),p_T(t))T_{(q_0',p_1')}
  F^{\calM}_t(\delta q_0',\delta p_1') dt,
\end{equation}
over $B^\infty_{r'}(q_0,p_1)$.

\begin{proposition}
  \label{prop:existence_of_exact_discrete_hamiltonian_system-local}
  With the definitions as above, $H_{d,T}^E(q_0',p_1')$ and
  $\phi_{d,T}^E(q_0',p_1')$ given
  by~\eqref{eq:exact_discrete_hamiltonian_system-II-hamiltonian-local}
  and~\eqref{eq:exact_discrete_hamiltonian_system-II-force-local} are
  smooth as functions of $T\in (T_-,T_+)$ and
  $(q_0',p_1')\in B^\infty_{r'}(q_0,p_1)$.
\end{proposition}

\begin{proof}
  It follows readily from
  Proposition~\ref{prop:flow_of_hamilton_equations-b}.
\end{proof}

Next, we explain the ``exact'' part of the name of $\calM_{d,h}^E$
introduced in Example~\ref{ex:exact_discrete_hamiltonian_system}. We
first have a technical result involving forced discrete Legendre
transforms and, then, the real goal, relating discrete trajectories of
$\calM_{d,h}^E$ to the continuous ones of $\calM$.

\begin{lemma}\label{le:legendre_transforms_for_exact_system}
  Let $\calM\CE (Q,H,\phi)$ be a forced Hamiltonian system and
  $\calM_{d,h}^E\CE (Q,H_{d,h}^E,\phi_{d,h}^E)$ be the exact forced
  discrete Hamiltonian system associated to $\calM$
  (Example~\ref{ex:exact_discrete_hamiltonian_system}). Then,
  \begin{equation*}
    \ff^+_{\phi_{d,h}^E} H_{d,h}^E(q_0,p_1) = (q(h),p_1)\stext{ and }
    \ff^-_{\phi_{d,h}^E} H_{d,h}^E(q_0,p_1) = (q_0,p(0)) 
  \end{equation*}
  where $(q(t),p(t)) \CE \FlowCBV{\calM}{h}{t}(q_0,p_1)$ is the
  boundary value flow of $\calM$ corresponding to the boundary values
  $(q_0,p_1)$.
\end{lemma}

\begin{proof}
  Let $(q(t),p(t)) \CE \FlowCBV{\calM}{h}{t}(q_0,p_1)$ be as in the
  statement. Then,
  differentiating~\eqref{eq:exact_discrete_hamiltonian_system-II-hamiltonian-local}
  with respect to $q_0'$ and integrating by parts, we obtain
  \begin{equation*}
    \begin{split}
      D_1H_{d,h}^E(q_0,p_1) =& 
      p(0) - \int_0^h \bigg(\big( -\dot{p}(t)- D_1H(q(t),p(t))\big)
      \pd{q(t)}{q_0} \\& \phantom{p(0) - \int_0^h \bigg(} +
      \big(\dot{q}(t)- D_2H(q(t),p(t))\big) \pd{p(t)}{q_0} \bigg) dt.
    \end{split}
  \end{equation*}
  Also, for any $X_{q_0} \in T_{q_0}Q$,
  \begin{equation*}
    (\phi_{d,h}^E)^q(q_0,p_1)(X_{q_0}) = \int_0^h
    \check{\phi}(q(t),p(t)) \left(\pd{q(t)}{q_0} X_{q_0}\right) dt.
  \end{equation*}
  Thus,
  \begin{gather*}
    D_1H_{d,h}^E(q_0,p_1) - (\phi_{d,h}^E)^q(q_0,p_1) =
    p(0) - \int_0^h \bigg( \big(\dot{q}(t)- D_2H(q(t),p(t))\big) \pd{p(t)}{q_0}
    \\+ \big( -\dot{p}(t)- D_1H(q(t),p(t)) +\check{\phi}(q_0,p_1)\big) \pd{q(t)}{q_0} \bigg) dt =
    p(0),
  \end{gather*}
  where the integral vanished because $(q(t),p(t))$ is a trajectory of
  $\calM$ and, so, it satisfies~\eqref{eq:Hamilton_equations}.

  The second formula in the statement follows along similar lines.
\end{proof}

\begin{proposition}\label{prop:exact_trajectories_imp_disc_trajectories}
  Given a forced Hamiltonian system $\calM\CE (Q,H,\phi)$ and $h>0$,
  let $\calM_{d,h}^E\CE (Q,H_{d,h}^E,\phi_{d,h}^E)$ be the exact
  forced discrete Hamiltonian system constructed in
  Example~\ref{ex:exact_discrete_hamiltonian_system}.
  \begin{enumerate}
  \item
    \label{it:exact_trajectories_imp_disc_trajectories-II-cont_are_disc}
    In addition, let $(q(t),p(t))$ be a trajectory of $\calM$ defined,
    at least, for $t\in [0,2h]$.  Then,
    $((q(0),p(0)),(q(h),p(h)),(q(2h),p(2h)))$ is an extended
    trajectory of $\calM_{d,h}^E$.
  \item
    \label{it:exact_trajectories_imp_disc_trajectories-II-disc_are_cont}
    Conversely, if $((q_0,p_0),(q_1,p_1),(q_2,p_2))$ is an extended
    discrete trajectory of $\calM_{d,h}^E$, there is a trajectory
    $(q(t),p(t))$ of $\calM$ such that $(q_k,p_k) = (q(kh),p(kh))$,
    for $k=0,1,2$.
  \end{enumerate}
\end{proposition}

\begin{proof}
  Let $(q_k,p_k)\CE (q(kh),p(kh))$ for $k=0,1,2$, where $(q(t),p(t))$
  is a trajectory of $\calM$, as in the statement. According to
  Remark~\ref{rem:discrete_Hamilton_equations-ff}, in order to prove
  that $(q_\cdot,p_\cdot)$ is a trajectory of $\calM_{d,h}^E$ we have
  to check that
  \begin{equation*}
    \ff^+_{\phi_{d,h}^E} H_{d,h}^E(q_{0},p_1) =
    \ff^-_{\phi_{d,h}^E} H_{d,h}^E(q_1,p_{2}).
  \end{equation*}
  Using Lemma~\ref{le:legendre_transforms_for_exact_system}, this
  identity becomes
  \begin{equation}
    \label{eq:exact_trajectories_imp_disc_trajectories-II-cond_1}
    (\conj{q}(h),p_1) = (q_1,\conj{p}(0)),
  \end{equation}
  where $\conj{q}(t)$ is the first component of
  $\FlowCBV{\calM}{h}{t}(q_0,p_1)$, and $\conj{p}$ is the second
  component of $\FlowCBV{\calM}{h}{t}(q_1,p_2)$. By construction, we
  see that $\conj{q}(t) = q(t)$ and $\conj{p}(t)=p(t+h)$. Thus,
  condition~\eqref{eq:exact_trajectories_imp_disc_trajectories-II-cond_1}
  is satisfied and $(q_\cdot,p_\cdot)$ is a trajectory of
  $\calM_{d,h}^E$.

  Similarly, using
  Lemma~\ref{le:legendre_transforms_for_exact_system}, writing
  $\conj{p}(t)$ for the second component of
  $\FlowCBV{\calM}{h}{t}(q_0,p_1)$, we have
  $(q_0,p_0) = (q_0,\conj{p}(0)) = \ff^-_{\phi_{d,h}^E}
  H_{d,h}^E(q_0,p_1)$, where we have used that $\conj{p}(t) =
  p(t)$. Last, writing $\conj{q}(t)$ for the first component of
  $\FlowCBV{\calM}{h}{t}(q_1,p_2)$, we have
  $(q_2,p_2) = (\conj{q}(h),p_2) = \ff^+_{\phi_{d,h}^E}
  H_{d,h}^E(q_{1},p_2)$, where we have used that
  $\conj{q}(t)=q(t+h)$. So, we conclude that $(q_\cdot,p_\cdot)$ is an
  extended trajectory of $\calM_{d,h}^E$, proving
  point~\eqref{it:exact_trajectories_imp_disc_trajectories-II-cont_are_disc}
  of the statement.

  Conversely, given an extended trajectory
  $(q_\cdot,p_\cdot) \CE ((q_0,p_0),(q_1,p_1),(q_2,p_2))$ of
  $\calM_{d,h}^E$ we define
  \begin{equation*}
    (q_0(t),p_0(t)) \CE \FlowCBV{\calM}{h}{t}(q_0,p_1) \stext{ and }
    (q_1(t),p_1(t)) \CE \FlowCBV{\calM}{h}{t}(q_1,p_2)
  \end{equation*}
  and observe that, because of
  Lemma~\ref{le:legendre_transforms_for_exact_system} and the fact
  that $(q_\cdot,p_\cdot)$ is a trajectory of $\calM_{d,h}^E$,
  \begin{equation}\label{eq:exact_trajectories_imp_disc_trajectories-II-BCs}
    (q_0(h),p_0(h)) = \ff^+_{\phi_{d,h}^E} H_{d,h}^E(q_0,p_1) =
    \ff^-_{\phi_{d,h}^E} H_{d,h}^E(q_1,p_2) = (q_1(0),p_1(0)).
  \end{equation}
  Now, both $(q_0(t),p_0(t))$ and $(q_1(t),p_1(t))$ are solutions of
  the same system of first order ODEs, albeit with different initial
  conditions. From this
  perspective,~\eqref{eq:exact_trajectories_imp_disc_trajectories-II-BCs}
  says that
  \begin{equation*}
    (q(t),p(t)) \CE 
    \begin{cases}
      (q_0(t),p_0(t)) \stext{ if } 0\leq t \leq h,\\
      (q_1(t-h),p_1(t-h)) \stext{ if } h\leq t \leq 2h\\
    \end{cases}
  \end{equation*}
  is a solution of the same ODE ---hence a trajectory of $\calM$---,
  satisfying, for example, $q(0)=q_0$ and $p(h)=p_1$. Thus
  \begin{equation*}
    q_0=q(0),\quad q_1=q_1(0)=q(h),\quad p_1=p_0(h)=p(h),\stext{ and }
    p_2=p_1(h)=p(2h).
  \end{equation*}
  Last, as $(q_\cdot,p_\cdot)$ is an extended trajectory of
  $\calM_{d,h}^E$,
  using~\eqref{eq:extended_forced_discrete_Hamiton_equations-ff} and
  Lemma~\ref{le:legendre_transforms_for_exact_system},
  \begin{gather*}
    (q_0,p_0) = \ff^-_{\phi_{d,h}^E} H_{d,h}^E(q_0,p_1) =
    (q_0,p_0(0)) = (q(0),p(0)),\\
    (q_2,p_2) = \ff^+_{\phi_{d,h}^E} H_{d,h}^E(q_1,p_2) =
    (q_1(h),p_2) = (q(2h),p(2h)).
  \end{gather*}
  Thus, we also have $p_0 = p(0)$ and $q_2=q(2h)$ which, together with
  the previous computations concludes the proof of
  point~\eqref{it:exact_trajectories_imp_disc_trajectories-II-disc_are_cont}
  in the statement.
\end{proof}


\section{Error analysis: approximations}
\label{sec:error_analysis-approximations}

Given a (continuous) forced Hamiltonian system $\calM$, we saw in
Section~\ref{sec:error_analysis-discrete_exact_systems} that the
discrete exact system $\calM_{d,h}^E$ has the very attractive property
of having its trajectories be the trajectories of $\calM$, evaluated
at discrete times that are multiples of $h$.

Still, the family $\calM_{d,h}^E$ is almost never constructible in
practice. Thus, we consider approximations of $\calM_{d,h}^E$ by
families of FDHSs $\calM_{d,h}$ ---known as \jdef{discretizations} of
$\calM$--- that, ideally, also approximate the trajectories of
$\calM_{d,h}^E$ which, in turn, interpolate the trajectories of
$\calM$.


\subsection{Contact order of maps}
\label{sec:contact_order_of_maps}

Here we introduce the notion of contact order for families of maps
between (open subsets of) finite-dimensional real vector spaces. In
what follows, $Q$ and $Q'$ are two such spaces. For any map
$f:Q\times\R\rightarrow Q'$, we define
$\conj{f}:Q\times \R\rightarrow Q'\times \R$ by
$\conj{f}(q,h)\CE(f(q,h),h)$. In what follows, we will consider smooth
maps, with respect to the natural structures of smooth manifold on $Q$
and $Q'$. Also, $\calW,\calW_a\subset Q\times \R$ are open
neighborhoods of $Q\times \{0\}$.

\begin{definition}\label{def:contact_order_global}
  Let $f_a:\calW_a\rightarrow Q'$ be smooth maps for $a=1,2$. We say
  that $f_2=f_1+\calO(h^{r+1})$ for $r\in\N$ if there is an open
  subset $U\subset\calW_1\cap\calW_2$ containing $Q\times\{0\}$ and a
  continuous function $\delta f:U\rightarrow Q'$ such that
  \begin{equation}\label{eq:contact_order_global}
    f_2(q,h)-f_1(q,h) = h^{r+1} \delta f(q,h) \stext{ for all } (q,h)\in U.
  \end{equation}
  If $f_2=f_1+\calO(h^{r+1})$ it is said that $f_1$ and $f_2$ have
  \jdef{contact of order $r$}.
\end{definition}

\begin{remark}
  More general notions of contact order for maps of $C^k$ type and
  between smooth manifolds are also considered in the literature. See,
  for example,~\cite{ar:cuell_patrick:2007:skew_critical_problems}
  and~\cite{ar:fernandez_graiffZurita_grillo:2021:error_analysis_of_forced_discrete_mechanical_systems}. Still,
  Definition~\ref{def:contact_order_global} suffices for the current
  context.
\end{remark}

\begin{remark}
  Notice that the notion of contact order is not strict, in the sense
  that if the contact order between two maps is $r$ it could also be
  $r'$ for some $r'>r$.
\end{remark}

The following result provides a convenient characterization for the
contact order.

\begin{proposition}\label{prop:contact_order_and_derivatives}
  Let $f_a:\calW_a\rightarrow Q'$ be smooth maps for $a=1,2$ and
  $r\in\N$ constant. Then, the following assertions are equivalent.
  \begin{enumerate}
  \item\label{it:contact_order_and_derivatives-order}
    $f_2=f_1+\calO(h^{r+1})$.
  \item\label{it:contact_order_and_derivatives-derivatives} For each
    $q\in Q$ there is an open subset $V_q\subset Q$ containing $q$
    such that
    \begin{equation*}
      D_2^jf_2(q',h)|_{h=0} = D_2^jf_1(q',h)|_{h=0} \stext{ for all }
      j=0,\ldots,r \stext{ and } q'\in V_q.
    \end{equation*}
  \end{enumerate}
\end{proposition}

Proposition~\ref{prop:contact_order_and_derivatives} is an application
of Taylor's Formula (with a parameter). See Lemma 2.24
in~\cite{ar:fernandez_graiffZurita_grillo:2021:error_analysis_of_forced_discrete_mechanical_systems}.

In the next two results, we use open neighborhoods
$\calW_a\subset Q\times\R$ and $\calW'_a\subset Q'\times\R$ of
$Q\times\{0\}$ and $Q'\times \{0\}$, respectively.

\begin{proposition}\label{prop:order_of_compositions-R_vector_space-adapted}
  Let $f_a:\calW_a\rightarrow Q'$ and $g_a:\calW'_a\rightarrow Q''$ be
  smooth maps for $a=1,2$, such that such that
  $\conj{f_a}(\calW_a)\subset \calW'_a$ (for $a=1,2$),
  $f_2=f_1+\calO(h^{r+1})$ and $g_2=g_1+\calO(h^{r+1})$. Then,
  $g_2\circ \conj{f_2} = g_1\circ \conj{f_1} +\calO(h^{r+1})$.
\end{proposition}

\begin{proof}
  This result can be proved comparing the derivatives
  $D_2^j(g_a\circ \conj{f_a})(q,h)\big|_{h=0}$ for $a=1$ and $2$
  (Proposition~\ref{prop:contact_order_and_derivatives}), using a
  recursive formula for $\pd{}{h^k} \left(g_a(f_a(q,h),h)\right)$ in
  terms of a function (independent of $a$) of
  $D_1^\alpha D_2^\beta g_a(f_a(q,h),h)$ and $D_2^\gamma(f(q,h))$ for
  $0\leq \alpha, \beta, \gamma\leq k$.

  Alternatively, this result is an adaptation of (a part of)
  Proposition 3 of~\cite{ar:cuell_patrick:2007:skew_critical_problems}
  to the case of vector spaces.
\end{proof}


\begin{proposition}\label{prop:order_of_inverses-R_vector_space-adapted_h=0}
  Let $f_a:\calW_a\rightarrow Q'$ be smooth maps for $a=1,2$ such that
  $f_a|_{Q\times\{0\}}:Q\times\{0\}\rightarrow Q'$ are
  diffeomorphisms. Then,
  \begin{enumerate}
  \item\label{it:order_of_inverses-R_vector_space-adapted_h=0-diffeo}
    there are open subsets $U_a\subset \calW_a$ and
    $U'_a\subset Q'\times\R$ containing $Q\times\{0\}$ and
    $Q'\times\{0\}$ respectively, such that
    $\conj{f_a}(U_a)\subset U'_a$ and $\conj{f_a}|_{U_a}^{U'_a}$ is a
    diffeomorphism (onto).
  \item \label{it:order_of_inverses-R_vector_space-adapted_h=0-inverses}
    There are smooth maps $g_a:U'_a\rightarrow Q$ such that
    $\conj{g_a}|^{U_a}\circ \conj{f_a} = id_{U_a}$.
  \item \label{it:order_of_inverses-R_vector_space-adapted_h=0-order}
    If $f_2=f_1+\calO(h^{r+1})$ for some $r\in\N$, then
    $g_2=g_1+\calO(h^{r+1})$.
  \end{enumerate}
\end{proposition}

\begin{proof}
  Using the Cartesian product structures and the fact that
  $f_a|_{Q\times\{0\}}:Q\times\{0\}\rightarrow Q'$ is a
  diffeomorphism, we see that $\conj{f_a}$ is a local diffeomorphism
  at each $(q,0)$. Then, as
  $f_a|_{Q\times\{0\}}:Q\times\{0\}\rightarrow Q'$ is a bijection
  between closed submanifolds of $\calW_a$ and $Q'\times\R$,
  statement~\eqref{it:order_of_inverses-R_vector_space-adapted_h=0-diffeo}
  follows from Theorem 1
  in~\cite{ar:cuell_patrick:2007:skew_critical_problems}\footnote{The proof
    of this result refers to~\cite{bo:lang-differential_manifolds}. An
    alternative reference is to follow Exercise 14 on p. 56
    of~\cite{bo:Guillemin-Pollack-differential_topology}.}.

  As $\conj{f_a}|_{U_a}^{U'_a}$ is a diffeomorphism, we can define
  $g_a \CE p_1\circ \left( \conj{f_a}|_{U_a}^{U'_a}\right)^{-1}$ and
  it is easy to check that it satisfies the condition that appears in
  statement~\eqref{it:order_of_inverses-R_vector_space-adapted_h=0-inverses}.
  
  Statement~\eqref{it:order_of_inverses-R_vector_space-adapted_h=0-order}
  can be proved comparing the derivatives
  $(D_2^jg_a)(f_a(q,h),h)\big|_{h=0}$ for $a=1$ and $2$
  (Proposition~\ref{prop:contact_order_and_derivatives}), using a
  recursive formula for $\pd{}{h^k} (D_2^jg_a)(f_a(q,h),h)$ in terms
  of a function (independent of $a$) of
  $(D_1^\alpha D_2^\beta g_a)(f_a(q,h),h)$ and $(D_2^\gamma f)(q,h)$
  for $0\leq \alpha, \gamma \leq k$, $0\leq \beta < k$ and for each
  $k\in\N$. 

  Alternatively,
  statement~\eqref{it:order_of_inverses-R_vector_space-adapted_h=0-order}
  is an adaptation of Proposition 4
  in~\cite{ar:cuell_patrick:2007:skew_critical_problems} to the case of
  vector spaces.
\end{proof}

When $\calW\subset Q\times \R$ is an open subset,
$T\calW = (TQ\oplus T\R)|_\calW$. Then, if $f:\calW\rightarrow Q'$ is
smooth, we consider $D_1f$ to be the restriction of $Tf$ to
$(TQ\oplus\{0\})|_\calW$. As
$(TQ\oplus\{0\})|_\calW \simeq \calW\times Q$ is an open subset in a
vector space, we can apply our order of contact notions to $D_1f$.

\begin{lemma}\label{le:order_of_derivative}
  Let $f_a:\calW_a \rightarrow \R$ be smooth maps for $a=1,2$ such
  that $f_2 = f_1 + \calO(h^{r+1})$. Then,
  $D_1f_a:(TQ\times\{0\})|_{\calW_a}\rightarrow \R$ ($a=1,2$) satisfy
  $D_1f_2 = D_1f_1 + \calO(h^{r+1})$.
\end{lemma}

\begin{proof}
  Using Proposition~\ref{prop:contact_order_and_derivatives} this
  computation can be done checking that the first $r$ derivatives with
  respect to $h$ of $D_1f_1$ and $D_1f_2$ at points with $h=0$ are the
  same. This, in turn, follows by applying $D_1$ to the
  $D_2^jf_1(q,0) = D_2^jf_2(q,0)$ (for $j=0,\ldots,r$), which is valid
  by the hypotheses and the Proposition.
\end{proof}


\subsection{Contact order of systems}
\label{sec:contact_order_of_systems}

Here we extend the notion of contact order between maps to contact
order between FDHSs.

\begin{definition}\label{def:contact_order_between_FDHS}
  Let $\calM_{d,h}^1\CE (Q,H_{d,h}^1,\phi_{d,h}^1)$ and
  $\calM_{d,h}^2\CE (Q,H_{d,h}^2,\phi_{d,h}^2)$ be two smooth
  $1$-parameter families of forced discrete Hamiltonian systems. They
  are said to have \jdef{contact of order $r$} if
  $H_{d,h}^2=H_{d,h}^1+\calO(h^{r+1})$ and
  $\phi_{d,h}^2=\phi_{d,h}^1+\calO(h^{r+1})$; in this case, we write
  $\calM_{d,h}^2=\calM_{d,h}^1+\calO(h^{r+1})$.
\end{definition}

\begin{proposition}
  \label{prop:contact_orde_of_forced_discrete_legendre_transforms}
  Let $\calM_{d,h}^1\CE (Q,H_{d,h}^1,\phi_{d,h}^1)$ and
  $\calM_{d,h}^2\CE (Q,H_{d,h}^2,\phi_{d,h}^2)$ be two smooth
  $1$-parameter families of forced discrete Hamiltonian systems such
  that $\calM_{d,h}^2=\calM_{d,h}^1+\calO(h^{r+1})$. Then, the
  corresponding forced discrete Legendre transforms have contact order
  $r$, that is,
  \begin{equation*}
    \ff^\pm_{\phi_d^2} H_d^2 = \ff^\pm_{\phi_d^1} H_d^1 + \calO(h^{r+1}).
  \end{equation*}
\end{proposition}

\begin{proof}
  It follows applying Lemma~\ref{le:order_of_derivative} to
  $H_{d,h}^2 = H_{d,h}^1 +\calO(h^{r+1})$ and using that
  $\phi_{d,h}^2 = \phi_{d,h}^1+ \calO(h^{r+1})$.
\end{proof}

The next result is the key to the error analysis of forced Hamiltonian
systems.

\begin{theorem}\label{th:contact_order_imp_flow_contact_order}
  Let $\calM_{d,h}^1\CE (Q,H_{d,h}^1,\phi_{d,h}^1)$ and
  $\calM_{d,h}^2\CE (Q,H_{d,h}^2,\phi_{d,h}^2)$ be two smooth
  $1$-parameter families of forced discrete Hamiltonian systems such
  that their forced discrete Legendre transforms
  $\ff^\pm_{\phi_{d,h}^a} H_{d,h}^a$ (for $a=1,2$) are diffeomorphisms
  when $h=0$.  If $\calM_{d,h}^2=\calM_{d,h}^1+ \calO(h^{r+1})$, then,
  the corresponding discrete flows $\FlowD{\calM_{d,h}^1}$ and
  $\FlowD{\calM_{d,h}^2}$ have contact order $r$, that is
  \begin{equation*}
    \FlowD{\calM_{d,h}^2} = \FlowD{\calM_{d,h}^1} + \calO(h^{r+1}).
  \end{equation*}
\end{theorem}

\begin{proof}
  By hypothesis, $\calM_{d,h}^2=\calM_{d,h}^1+ \calO(h^{r+1})$, so
  that, by
  Proposition~\ref{prop:contact_orde_of_forced_discrete_legendre_transforms},
  $\ff^\pm_{\phi_{d,h}^2} H_{d,h}^2 = \ff^\pm_{\phi_{d,h}^1} H_{d,h}^1
  + \calO(h^{r+1})$. As both $\ff^\pm_{\phi_{d,h}^1} H_{d,h}^1$ and
  $\ff^\pm_{\phi_{d,h}^2} H_{d,h}^2$ are diffeomorphisms when $h=0$,
  by
  Proposition~\ref{prop:order_of_inverses-R_vector_space-adapted_h=0},
  we see that
  $(\ff^\pm_{\phi_{d,h}^2} H_{d,h}^2)^{-1} = (\ff^\pm_{\phi_{d,h}^1}
  H_{d,h}^1)^{-1} + \calO(h^{r+1})$. Then, by
  Proposition~\ref{prop:order_of_compositions-R_vector_space-adapted},
  we have that
  \begin{equation*}
    \ff^+_{\phi_{d,h}^2} H_{d,h}^2 \circ (\ff^-_{\phi_{d,h}^2} H_{d,h}^2)^{-1} =
    \ff^+_{\phi_{d,h}^1} H_{d,h}^1 \circ (\ff^-_{\phi_{d,h}^1} H_{d,h}^1)^{-1} +
    \calO(h^{r+1}).
  \end{equation*}
  The statement now follows
  from~\eqref{eq:discrete_flow_in_terms_of_legendre_transforms}.
\end{proof}


\subsection{Discretizations of FDHSs}
\label{sec:error_analysis-discretizations_of_FDHSs}

Now we come back to the relationship between a family of forced
discrete Hamiltonian systems and a given, continuous, one.

\begin{definition}\label{def:discretization}
  Given a forced Hamiltonian system $\calM\CE (Q,H,\phi)$, a smooth
  $1$-parameter family of forced discrete Hamiltonian systems
  $\calM_{d,h}\CE (Q,H_{d,h},\phi_{d,h})$ is a \jdef{discretization}
  of $\calM$ if it satisfies
  \begin{equation}\label{eq:discretization-conditions}
    \begin{gathered}
      H_{d,h}(q_0,p_1) = p_1q_0 + h H(q_0,p_1) + \calO(h^2),\\
      \phi_{d,h}(q_0,p_1) = h \phi(q_0,p_1) + \calO(h^2).
    \end{gathered}
  \end{equation}
\end{definition}

\begin{example}\label{ex:exact_discretizations_are_discretizations}
  Let $\calM\CE (Q,H,\phi)$ be a forced Hamiltonian system and
  $\calM^E_{d,h}\CE (Q,H^E_{d,h},\phi^E_{d,h})$ be the $1$-parameter
  family of exact discrete Hamiltonian systems associated to $\calM$
  constructed in
  Example~\ref{ex:exact_discrete_hamiltonian_system}. As
  $H^E_{d,h}$ and $\phi^E_{d,h}$ are smooth functions of all of their
  variables, including $h$ (near $0$), a careful Taylor expansion
  around $h=0$ shows that
  \begin{equation*}
    \begin{split}
      H_{d,h}^E(q_0,p_1) =& p_1q_0 + h H(q_0,p_1) + \calO(h^2),\\
      \phi_{d,h}^E(q_0,p_1) =& h \check{\phi}(q_0,p_1) + \calO(h^2).
    \end{split}
  \end{equation*}
  We see that $\calM^E_{d,h}$ is a discretization of $\calM$ that we
  naturally call the \jdef{exact discretization} of $\calM$. In fact,
  a longer and more careful computation shows that
  \begin{equation*}
    \begin{split}
      H_{d,h}^E(q_0,p_1) = p_1q_0 + H(q_0,p_1) h + \frac{1}{2} D_1H(q_0,p_1) D_2H(q_0,p_1) h^2 + \calO(h^3)
    \end{split}
  \end{equation*}
  and
  \begin{equation*}
    \begin{split}
      \phi_{d,h}^E(q_0,p_1)&(\delta q_0,\delta p_1) =
                                                     \check{\phi}(q_0,p_1)(\delta q_0) h + \frac{1}{2} \bigg(\ip{D_1(\check{\phi}(q_0,p_1)(\delta q_0))}{D_2H(q_0,p_1)}  \\& \phantom{ = } + \ip{\check{\phi}(q_0,p_1)}{D_{21}H(q_0,p_1)(\delta q_0) +D_{22}H(q_0,p_1)(\delta p_1)}
      \\& \phantom{ = } + \ip{D_2(\check{\phi}(q_0,p_1)(\delta q_0))}{D_1H(q_0,p_1)-\check{\phi}(q_0,p_1)}\bigg) h^2 + \calO(h^3).
    \end{split}
  \end{equation*}
  Alternatively, using local coordinates,
  \begin{equation}\label{eq:H_d^E-order_2}
    H_{d,h}^E(q_0,p_1) = p_{1,j}q_0^j + H(q_0,p_1) h + \frac{h^2}{2} \pd{H(q_0,p_1)}{q_0^j} \pd{H(q_0,p_1)}{p_{1,j}} + \calO(h^3)
  \end{equation}
  and
  \begin{equation}\label{eq:phi_d^E-order_2}
    \begin{split}
      \phi_{d,h}^E(q_0,p_1) =& h \sum_{k=1}^n \check{\phi}_k(q_0,p_1) dq_0^k + \frac{h^2}{2} \bigg[ \check{\phi}_j(q_0,p_1)  \frac{\partial^2 H(q_0,p_1)}{\partial p_{1,j} \partial p_{1,k}} d p_{1,k} \\&+\bigg( \pd{\check{\phi}_k(q_0,p_1)}{q_0^j} \pd{H(q_0,p_1)}{p_{1,j}} + \check{\phi}_j(q_0,p_1) \frac{\partial^2 H(q_0,p_1)}{\partial p_{1,j} \partial q_0^k}\bigg) dq_0^k    \bigg] + \calO(h^3). 
    \end{split}
  \end{equation}
\end{example}

\begin{definition}
  A discretization, $\calM_{d,h}$, of the forced Hamiltonian system
  $\calM\CE (Q,H,\phi)$ is said to have \jdef{contact of order $r$} if
  $\calM_{d,h} = \calM_{d,h}^E + \calO(h^{r+1})$, where
  $\calM_{d,h}^E$ is the exact discretization of $\calM$. Notice that
  all discretizations of $\calM$ have contact order $\geq 1$.
\end{definition}

We can apply Theorem~\ref{th:contact_order_imp_flow_contact_order}
to discretizations as follows.

\begin{theorem}
  \label{th:contact_order_imp_flow_contact_order-discretizations}
  Given a forced Hamiltonian system $\calM\CE (Q,H,\phi)$ let
  $\calM_{d,h}$ be any discretization of $\calM$ of contact order
  $r$. Then
  \begin{equation*}
    \FlowC{\calM}{h} = \FlowD{\calM_{d,h}^E} = \FlowD{\calM_{d,h}} +
    \calO(h^{r+1}).
  \end{equation*}
\end{theorem}

\begin{proof}
  The first identity is valid by
  Proposition~\ref{prop:exact_trajectories_imp_disc_trajectories}. By
  Lemma~\ref{le:legendre_transforms_for_exact_system} with $h=0$, we
  have $\ff^\pm_{\phi_{d,0}^E} H_{d,0}^E = id_{T^*Q}$ and, using
  Proposition~\ref{prop:contact_orde_of_forced_discrete_legendre_transforms}
  together with $\calM_{d,h} = \calM_{d,h}^E + \calO(h^{r+1})$, we see
  that $\ff^\pm_{\phi_{d,0}} H_{d,0} = id_{T^*Q}$. Then, the second
  identity in the statement follows from
  Theorem~\ref{th:contact_order_imp_flow_contact_order}, because both
  discrete Legendre transforms are diffeomorphisms when $h=0$ and
  $\calM_{d,h} = \calM_{d,h}^E +\calO(h^{r+1})$.
\end{proof}


\section{Construction of forced discrete Hamiltonian systems}
\label{sec:construction_of_FDHS}

In this section, we consider the practical construction of forced
discrete Hamiltonian systems $\calM_{d,h}\CE (Q,H_{d,h},\phi_{d,h})$
that approximate a given forced Hamiltonian system
$\calM\CE (Q,H,\phi)$. There are different possible approaches. One of
them relies on the approximations of orders $1$ and $2$ of the exact
discrete system $\calM_{d,h}^E$ provided by
Example~\ref{ex:exact_discretizations_are_discretizations}. More
precisely, defining
\begin{equation}\label{eq:H_d^E_and_phi_d^E-order_1}
  H_{d,h}(q_0,p_1) = p_1q_0 + h H(q_0,p_1) \stext{ and }
  \phi_{d,h}(q_0,p_1) = h \check{\phi}(q_0,p_1)
\end{equation}
we have a discretization of $\calM$ that is accurate to order $1$,
whereas defining $H_{d,h}$ and $\phi_{d,h}$
using~\eqref{eq:H_d^E-order_2} and~\eqref{eq:phi_d^E-order_2} we have
a discretization of $\calM$ that is accurate to order $2$.

Next we describe an alternative, using the shooting method,
following~\cite{ar:leok_shingel-general_techniques_for_constructing_variational_integrators}
and~\cite{ar:schmitt_shingel_leok-lagrangian_and_hamiltonian_taylor_variational_integrators}.
Both $H_{d,h}^E$ and $\phi_{d,h}^E$ are defined
in~\eqref{eq:exact_discrete_hamiltonian_system-II-hamiltonian}
and~\eqref{eq:exact_discrete_hamiltonian_system-II-force} as integrals
of certain functions that involve the exact trajectories of
$\calM$. Thus, one can use a quadrature formula to approximate the
integral and a numerical integrator for the (initial value problem)
that approximates the exact trajectory of $\calM$. More precisely, fix
a Gaussian quadrature formula such that, for $f:[0,h]\rightarrow \R$,
it is accurate to order $a$, that is,
\begin{equation}\label{eq:shooting-quadrature_precission}
  \sum_{j=0}^n b_j f(c_jh) = \int_0^h f(x) dx +\calO(h^{a+1}).
\end{equation}
We also fix a numerical integrator $\Phi_h:T^*Q\rightarrow T^*Q$ for
Hamilton's equations~\eqref{eq:Hamilton_equations} that is accurate to
order $b$, that is,
\begin{equation}\label{eq:shooting-integrator_precission}
  \Phi_h(q,p) = \FlowC{\calM}{h}(q,p) + \calO(h^{b+1}).
\end{equation}

Ideally, we would proceed as follows: in order to define
$H_{d,h}(q_0,p_1)$ we would find $p_0\in Q^*$ such that
$\pr_2(\FlowC{\calM}{h}(q_0,p_0))=p_1$. Then, we would use
$\Phi_{c_j h}(q_0,p_0)$ to define a discrete trajectory (that
approximates $\FlowC{\calM}{c_j h}(q_0,p_0)$) and, last, use the
quadrature formula
and~\eqref{eq:exact_discrete_hamiltonian_system-II-hamiltonian} to
define $H_{d,h}(q_0,p_1)$. Unfortunately, we don't know
$\FlowC{\calM}{h}(q_0,p_0)$, so that we replace $p_0$ by
$\ti{p_0}\in Q^*$ such that $\pr_2(\Phi_h(q_0,\ti{p_0})) = p_1$. It can be
seen that $\ti{p_0} = p_0 +\calO(h^{b+1})$ (this is analogous to Lemma
3.1
in~\cite{ar:schmitt_shingel_leok-lagrangian_and_hamiltonian_taylor_variational_integrators}). Then, we define
\begin{equation*}
  (q^j,p^j)\CE \Phi_{c_j h}(q_0,\ti{p_0}) \stext{ and } v^j \CE D_2H(q^j,p^j)
  \stext{ for } j=0,\ldots,n.
\end{equation*}
Notice that, for $(q(t),p(t)) \CE \FlowC{\calM}{t}(q_0,p_0)$,
\begin{equation*}
  \begin{split}
    (q(c_j h),p(c_j h)) =& \FlowC{\calM}{t}(q_0,p_0) = \FlowC{\calM}{t}(q_0,\ti{p_0} + \calO(h^{b+1})) =  \FlowC{\calM}{t}(q_0,\ti{p_0}) + \calO(h^{b+1}) \\=& \Phi_{c_j h}(q_0,\ti{p_0}) +\calO(h^{b+1}) = (q^j,p^j) +\calO(h^{b+1}),
  \end{split}
\end{equation*}
where the third equality follows by a first order Taylor expansion of
the flow and the fourth is due to the accuracy of the numerical
integrator. On the other hand, as $(q(t),p(t))$
solves~\eqref{eq:Hamilton_equations}, we have
\begin{equation*}
  \begin{split}
    \dot{q}(c_j h) =& D_2H(q(c_j h),p(c_j h)) = D_2H((q^j,p^j)+\calO(h^{b+1})) \\=& D_2H(q^j,p^j) + \calO(h^{b+1}) = v^j +\calO(h^{b+1}).
  \end{split}
\end{equation*}
Last, we define
\begin{equation}\label{eq:shooting_H_phi-def}
  \begin{split}
    \begin{gathered}[t]
      H_{d,h}(q_0,p_1) \CE p^nq^n-\sum_{j=0}^n b_j(p^jv^j-H(q^j,p^j)),\\
      \phi_{d,h}(q_0,p_1)(\delta q_0,p_1) \CE \sum_{j=1}^n b_j\check{\phi}(q^j,p^j)(T_{(q_0,p_1)}q^j(\delta q_0,\delta p_1)).
    \end{gathered}
  \end{split}
\end{equation}
Notice that, by construction, all $(q^j,p^j)$ are functions of $q_0$
and $p_1$ and, as long as $\Phi_h$ is $C^1$, the derivative
$T_{(q_0,p_1)}q^j$ that appears in $\phi_{d,h}$ is well defined. In
addition, it follows from $q^j = q(c_j h) +\calO(h^{b+1})$ that
$T_{(q_0,p_1)}q^j = T_{(q_0,p_1)}q(c_j h) + \calO(h^{b+1})$.

\begin{proposition}\label{prop:shooting_H_phi_precission}
  For the maps $H_{d,h}$ and $\phi_{d,h}$ constructed
  using~\eqref{eq:shooting_H_phi-def} for a quadrature rule
  satisfying~\eqref{eq:shooting-quadrature_precission} and a numerical
  integrator $\Phi_h$
  satisfying~\eqref{eq:shooting-integrator_precission}, we have
  \begin{equation*}
    H_{d,h} = H_{d,h}^E +\calO(h^{\min\{a,b\}+1}) \stext{ and }
    \phi_{d,h} = \phi_{d,h}^E +\calO(h^{\min\{a,b\}+1}).
  \end{equation*}
\end{proposition}

\begin{proof}
  We have
  \begin{equation*}
    \begin{split}
      H_{d,h}^E&(q_0,p_1) = p(h) q(h) - \int_0^h (p(t) \dot{q}(t) -H(q(t),p(t))) dt \\=& p(h) q(h) - \left( \sum_{j=0}^n b_j \left( p(c_j h) \dot{q}(c_j h) - H(q(c_j h), p(c_j h))\right)\right) + \calO(h^{a+1})  \\=& (p^n+\calO(h^{b+1})) (q^n+\calO(h^{b+1})) - \bigg(\sum_{j=0}^n b_j ((p^j+\calO(h^{b+1})) (v^j+\calO(h^{b+1})) \\&\phantom{(p^n+\calO(h^{b+1})) (q^n+\calO(h^{b+1})) - \bigg(} - \underbrace{H((q^j,p^j)+\calO(h^{b+1}))}_{=H(q^j,p^j) +\calO(h^{b+1})})\bigg)  +\calO(h^{a+1}) \\=& p^n q^n- \left(\sum_{j=0}^n b_j (p^j v^j - H(q^j,p^j))\right) + \calO(h^{a+1}) + \calO(h^{b+1}) \\=& H_{d,h}(q_0,p_1) +\calO(h^{\min\{a,b\}+1}).
    \end{split}
  \end{equation*}
  A similar computation shows the result for $\phi_{d,h}$.
\end{proof}

The following result, whose proof follows by unraveling the
definitions, shows that the order $1$ discretization considered at the
beginning of the section has some interesting properties in connection
with the commutativity of the discretization and the passage from the
Hamiltonian to the Lagrangian formalisms.

\begin{proposition}\label{prop:discretization and 1st order approximation}
  Let $(Q,H,\phi)$ be a regular forced Hamiltonian system of
  mechanical type with potential depending only on $q$. Given a fixed
  time--step $h>0$, the FDHS $(Q,H_{d,h},\phi_{d,h})$ given
  by~\eqref{eq:H_d^E_and_phi_d^E-order_1}, is equivalent to the system
  obtained via the FDLS constructed using the discretization
  $\Delta_h : TQ \lra Q \times Q$,
  \begin{equation*}
    \Delta_h(q,v) \CE  \left( q , q + h v \right), \quad
    \Delta_h^{-1}(q_0,q_1) = \left( q_0 , \frac{q_1 - q_0}{h} \right).    
  \end{equation*}
  In other words, the following diagram commutes
  \begin{equation*}
    \xymatrix{
      (Q,H,\phi) \ar@{-->}[d] \ar[r]^-{\ff H} &
      (Q,L,f) \ar@{-->}[d]^-{\Delta_h} \\
      (Q,H_{d,h},\phi_{d,h}) & (Q,L_d,f_d) \ar[l]_-{\widetilde{\ff}_{f_d} L_d}
    }
  \end{equation*}
  where the arrows represent the different constructions using the
  mentioned maps and the dashed ones indicate that it results in a
  discrete system.
\end{proposition}

\begin{example}\label{ex:sextic_potential}
  Let us revisit
  Example~\ref{example:unit_mass_particle_with_radial_potential},
  where we considered the forced Hamiltonian system
  $\calM\CE(Q,H,\phi)$ given by $Q \CE \R^2$,
  \begin{equation*}
    H(q,p) \CE  \frac{\norm{p}^2}{2} + \norm{q}^2 \left( \norm{q}^2 - 1 \right)^2 \stext{ and } \phi(q,p) \CE  -\mu \left( p^x \ dq^x + p^y \ dq^y \right),
  \end{equation*}
  where $\mu\geq 0$ is a constant and we are using the notation
  $q = (q^x,q^y)$, $p = (p^x,p^y)$.

  Using the order $1$ discretization of $\calM$ given
  by~\eqref{eq:H_d^E_and_phi_d^E-order_1}, we obtain the FDHS
  $\calM_{d,h}\CE(Q,H_{d,h},\phi_{d,h})$ with
  \begin{gather*}
    H_{d,h}(q,p) = p q + h H(q,p) =
    \langle p,q \rangle + \frac{h}{2} \norm{p}^2 +
    h \norm{q}^2 \left( \norm{q}^2 - 1 \right)^2,\\
    \phi_{d,h}(q,p) = h \check{\phi}(q,p) =
    -\mu h \left( p^x \ dq^x + p^y \ dq^y \right).
  \end{gather*}

  Using the equations~\eqref{eq:extended_discrete_Hamilton_equations},
  we find that the flow of $\calM_{d,h}$ is
  \begin{align*}
      q^x_{k+1}& = q^x_k + \frac{h}{1 + h\mu} \left( p^x_k - h q^x_k K(q_k) \right),&  p^x_{k+1}& = \frac{1}{1 + h\mu} \left( p^x_k - h q^x_k K(q_k) \right),\\
      q^y_{k+1}& = q^y_k + \frac{h}{1 + h\mu} \left( p^y_k - h q^y_k K(q_k) \right),& p^y_{k+1}& = \frac{1}{1 + h\mu} \left( p^y_k - 2 h q^y_k K(q_k) \right),\\
  \end{align*}
  where
  $K(q)\CE 2\left((\norm{q}^2 -1)^2 + 2 \norm{q}^2 (\norm{q}^2 -
    1)\right)$.

  This FDHS gives rise to an order $1$ numerical integrator that we
  call \texttt{FDHS-O1}. A similar construction starting with $\calM$
  and using~\eqref{eq:H_d^E-order_2} and~\eqref{eq:phi_d^E-order_2}
  leads to an order $2$ numerical integrator that we call
  \texttt{FDHS-O2}. We compare these integrators with the classical
  order $4$ Runge--Kutta denoted by \texttt{RK4} (applied to the
  equations~\eqref{eq:Hamilton_equations} corresponding to
  $\calM$). As a benchmark we use the solution provided by
  \texttt{NDSolve} in \text{Mathematica 12}.

  Following Example 3.14 in
  \cite{ar:caruso_fernandez_tori_zuccalli:2023:lagrangian_reduction_of_forced_discrete_mechanical_systems},
  we consider $\mu \CE 10^{-3}$ and initial conditions
  $q_0 \CE (0.1,1.1)$, $p_0 \CE (0.6,0.1)$ to plot the trajectories
  and the energy evolution of the system. In addition, we used a
  timestep of $h=0.2$ and a maximum time of
  $4000$. Figure~\ref{fig:Q_exact-q} shows the (benchmark) evolution
  of the system in $Q$. Figure~\ref{fig:Q_exact-q_error} compares the
  local error for the \texttt{FDHS-O1}, \texttt{FDHS-O2} and the
  \texttt{RK4} integrators: notice that all show approximately the
  same order of magnitude in the error. Still, the error for
  \texttt{FDHS-O2} is about one third of the error of \texttt{RK4}
  even though one is order $2$ and the other is order $4$ ---this is
  possibly due to the fact that $h$ is fairly large, which is
  convenient for the applications.
  \begin{figure}[H]
    \centering \subfigure[Benchmark position
    evolution\label{fig:Q_exact-q}]{\includegraphics[scale=1]{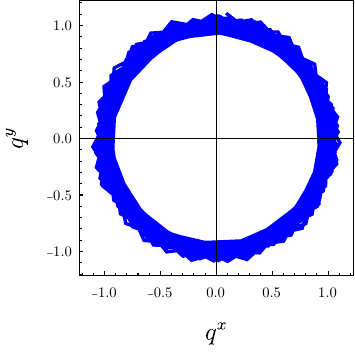}}
    \goodgap \subfigure[Error in position
    evolution\label{fig:Q_exact-q_error}]{\includegraphics[scale=1]{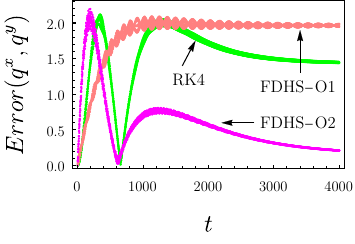}}
    \caption{Position evolution (Benchmark and numerical). Parameters:
      $\mu=10^{-3}$ and $h=0.2$. Initial conditions: $q_0=(0.1,1.1)$,
      $p_0=(0.6,0.1)$. Methods: Benchmark (Blue), \texttt{FDHS-O1}
      (Pink), \texttt{FDHS-O2} (Magenta) and \texttt{RK4} (Green).}
    \label{fig:Q_exact}
  \end{figure}

  As this is a forced system, with a friction-type force, it is
  interesting to compare the energy evolution under the different
  numerical integrators. Figure~\ref{fig:H_evolution-exact} shows the
  (benchmark) evolution of the energy of the
  system. Figure~\ref{fig:H_evolution-numerical} compares the errors
  of the energy estimates for \texttt{FDHS-O1}, \texttt{FDHS-O2} and
  \texttt{RK4}. Notice that, qualitatively speaking, all three
  integrators follow the benchmark, that is, the error is decreasing;
  the graph also shows the typical ``oscillatory'' behavior of the
  error for the variational integrators. Last, observe how, even the
  order $1$ \texttt{FDHS-O1} outperforms \texttt{RK4} in this case.

    \begin{figure}[H]
      \centering \subfigure[Benchmark energy
      evolution\label{fig:H_evolution-exact}]{\includegraphics[scale=1]{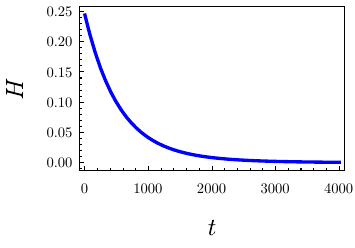}}
      \goodgap \subfigure[Error in energy
      evolution\label{fig:H_evolution-numerical}]{\includegraphics[scale=1]{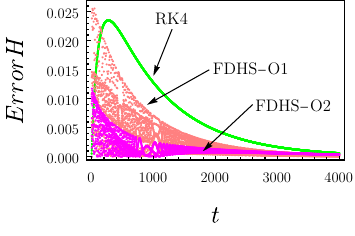}}
    \caption{Energy evolution (Benchmark and numerical). Parameters:
      $\mu=10^{-3}$ and $h=0.2$. Initial conditions: $q_0=(0.1,1.1)$,
      $p_0=(0.6,0.1)$. Methods: Benchmark (Blue), \texttt{FDHS-O1} (Pink),
      \texttt{FDHS-O2} (Magenta) and \texttt{RK4} (Green).}
    \label{fig:H_evolution}
  \end{figure}

\end{example}


\section*{Acknowledgments}

This document is the result of research partially supported by grants
from the Universidad Nacional de Cuyo [code 06/80020240100069UN],
Universidad Nacional de La Plata [code SX007], and CONICET.


\printbibliography

@article{ar:caruso_fernandez_tori_zuccalli:2023:lagrangian_reduction_of_forced_discrete_mechanical_systems,
author = {Matías I {Caruso} and Javier {Fernández} and Cora {Tori} and Marcela {Zuccalli}},
title = {Lagrangian reduction of forced discrete mechanical systems},
journal = {Journal of Physics A: Mathematical and Theoretical},
doi = {10.1088/1751-8121/aceae3},
year = {2023},
publisher = {IOP Publishing},
volume = {56},
number = {35},
pages = {355202},
}

@article{ar:cuell_patrick:2007:skew_critical_problems,
    AUTHOR = {Charles {Cuell} and George. W. {Patrick}},
     TITLE = {Skew critical problems},
   JOURNAL = {Regul. Chaotic Dyn.},
  FJOURNAL = {Regular and Chaotic Dynamics. International Scientific
              Journal},
    VOLUME = {12},
      YEAR = {2007},
    NUMBER = {6},
     PAGES = {589--601},
      ISSN = {1560-3547},
   MRCLASS = {37J60 (70F25 70H30 90C30)},
  MRNUMBER = {2373159 (2010g:37100)},
       DOI = {10/dvjwd3},
       URL = {https://doi.org/dvjwd3},
}

@article{ar:fernandez_graiffZurita_grillo:2021:error_analysis_of_forced_discrete_mechanical_systems,
  author = 	 {Javier {Fern\'andez} and Sebasti\'an {Graiff Zurita} and
                  Sergio {Grillo}},
  title = 	 {Error analysis of forced discrete mechanical systems},
  journal = 	 {J.Geom. Mech.},
  year = 	 {2021},
  OPTkey = 	 {},
  volume = 	 {13},
  OPTnumber = 	 {},
  OPTpages = 	 {533-606},
  OPTmonth = 	 {},
  note = 	 {Preprint available:
                  \href{http://arXiv.org/abs/2103.11060}{\tt
                  arXiv:2103.11060}},
  OPTannote = 	 {},
  DOI = {10.3934/jgm.2021017},
  URL = {https://dx.doi.org/10.3934/jgm.2021017}
}

@article {ar:leok_shingel-general_techniques_for_constructing_variational_integrators,
    AUTHOR = {Leok, Melvin and Shingel, Tatiana},
     TITLE = {General techniques for constructing variational integrators},
   JOURNAL = {Front. Math. China},
  FJOURNAL = {Frontiers of Mathematics in China},
    VOLUME = {7},
      YEAR = {2012},
    NUMBER = {2},
     PAGES = {273--303},
      ISSN = {1673-3452},
   MRCLASS = {65P10 (70H25)},
  MRNUMBER = {2897705},
MRREVIEWER = {Fernando Casas},
       DOI = {10.1007/s11464-012-0190-9},
       URL = {http://dx.doi.org/10.1007/s11464-012-0190-9},
}

@article{ar:leok_zhang:2011:discrete_hamiltonian_variational_integrators,
    AUTHOR = {Melvin {Leok} and Jingjing {Zhang}},
     TITLE = {Discrete {H}amiltonian variational integrators},
   JOURNAL = {IMA J. Numer. Anal.},
  FJOURNAL = {IMA Journal of Numerical Analysis},
    VOLUME = {31},
      YEAR = {2011},
    NUMBER = {4},
     PAGES = {1497--1532},
      ISSN = {0272-4979},
     CODEN = {IJNADN},
   MRCLASS = {65P10 (37M15)},
  MRNUMBER = {2846764},
MRREVIEWER = {Fernando Casas},
       DOI = {10.1093/imanum/drq027},
       URL = {http://dx.doi.org/10.1093/imanum/drq027},
}

@article {ar:marsden_west:2001:discrete_mechanics_and_variational_integrators,
    AUTHOR = {Jerrold E. {Marsden} and M. {West}},
     TITLE = {Discrete mechanics and variational integrators},
   JOURNAL = {Acta Numer.},
  FJOURNAL = {Acta Numerica},
    VOLUME = {10},
      YEAR = {2001},
     PAGES = {357--514},
      ISSN = {0962-4929},
       DOI = {10.1017/S096249290100006X},
   MRCLASS = {37M15 (37J05 65P10 70-08 70H05)},
  MRNUMBER = {MR2009697 (2004h:37130)},
MRREVIEWER = {Christian Lubich},
}

@article{ar:schmitt_leok:2017:properties_of_hamiltonian_variational_integrators,
	author = {Jeremy M {Schmitt} and Melvin {Leok}},
	title = "{Properties of Hamiltonian variational integrators}",
	journal = {IMA Journal of Numerical Analysis},
	volume = {38},
	number = {1},
	pages = {377-398},
	year = {2017},
	issn = {0272-4979},
	doi = {10.1093/imanum/drx010},
	url = {https://doi.org/10.1093/imanum/drx010},
}

@book {bo:goldstein:classical_mechanics,
    AUTHOR = {Herbert {Goldstein}},
     TITLE = {Classical mechanics},
   EDITION = {Second Edition},
      NOTE = {Addison-Wesley Series in Physics},
 PUBLISHER = {Addison-Wesley Publishing Co., Reading, Mass.},
      YEAR = 1980,
     PAGES = {xiv+672},
      ISBN = {0-201-02918-9},
   MRCLASS = {70-02},
  MRNUMBER = {575343 (81j:70001)},
MRREVIEWER = {Ll. G. Chambers},
}

@book {bo:teschl-ordinary_differential_equations_and_dynamical_systems,
    AUTHOR = {Teschl, Gerald},
     TITLE = {Ordinary differential equations and dynamical systems},
    SERIES = {Graduate Studies in Mathematics},
    VOLUME = {140},
 PUBLISHER = {American Mathematical Society},
   ADDRESS = {Providence, RI},
      YEAR = {2012},
     PAGES = {xii+356},
      ISBN = {978-0-8218-8328-0},
  language = {english},
   MRCLASS = {34-01 (37-01 39-01)},
  MRNUMBER = {2961944},
MRREVIEWER = {Eleonora Catsigeras},
}

@book {bo:keller-numerical_methods_for_two_point_boundary_value_problems,
    AUTHOR = {Keller, Herbert B.},
     TITLE = {Numerical methods for two-point boundary value problems},
      NOTE = {Corrected reprint of the 1968 edition},
 PUBLISHER = {Dover Publications, Inc., New York},
      YEAR = {1992},
     PAGES = {x+397},
      ISBN = {0-486-66925-4},
  language = {english},
   MRCLASS = {65L10 (34A50 34B15)},
  MRNUMBER = {1207811 (93k:65066)},
}

@book {bo:Guillemin-Pollack-differential_topology,
    AUTHOR = {Guillemin, Victor and Pollack, Alan},
     TITLE = {Differential topology},
 PUBLISHER = {Prentice-Hall Inc.},
   ADDRESS = {Englewood Cliffs, N.J.},
      YEAR = {1974},
     PAGES = {xvi+222},
  language = {english},
      ISBN = {0-13-212605-2},
   MRCLASS = {58-01 (57-01)},
  MRNUMBER = {MR0348781 (50 \#1276)},
MRREVIEWER = {K. H. Mayer},
 zbMATH = {3562121},
 Zbl = {0361.57001},
}

@book {bo:lang-differential_manifolds,
    AUTHOR = {Lang, Serge},
     TITLE = {Differential manifolds},
 PUBLISHER = {Addison-Wesley Publishing Co., Inc., Reading, Mass.-London-Don
              Mills, Ont.},
      YEAR = {1972},
     PAGES = {ix+230},
  language = {english},
   MRCLASS = {58-01},
  MRNUMBER = {0431240 (55 \#4241)},
MRREVIEWER = {N. J. Hicks},
}

@article{ar:schmitt_shingel_leok-lagrangian_and_hamiltonian_taylor_variational_integrators,
 author = {Schmitt, Jeremy and Shingel, Tatiana and Leok, Melvin},
 title = {Lagrangian and {Hamiltonian} {Taylor} variational integrators},
 fjournal = {BIT},
 journal = {BIT},
 issn = {0006-3835},
 volume = {58},
 number = {2},
 pages = {457--488},
 year = {2018},
 language = {english},
 doi = {10.1007/s10543-017-0690-9},
 zbMATH = {6894732},
 Zbl = {1392.37089}
}

@book {bo:hairer_lubich_wanner-geometric_numerical_integration,
    AUTHOR = {Hairer, Ernst and Lubich, Christian and Wanner, Gerhard},
     TITLE = {Geometric numerical integration},
    SERIES = {Springer Series in Computational Mathematics},
    VOLUME = {31},
   EDITION = {Second},
      NOTE = {Structure-preserving algorithms for ordinary differential
              equations},
 PUBLISHER = {Springer-Verlag},
   ADDRESS = {Berlin},
      YEAR = {2006},
     PAGES = {xviii+644},
      ISBN = {978-3-540-30663-4},
   MRCLASS = {65-02 (37M15 65Lxx 65P10 70-08)},
  MRNUMBER = {MR2221614 (2006m:65006)},
}

@book {bo:AM-mechanics,
    AUTHOR = {Abraham, Ralph and Marsden, Jerrold E.},
     TITLE = {Foundations of mechanics},
      NOTE = {Second edition, revised and enlarged,
              With the assistance of Tudor Ra{\c{t}}iu and Richard Cushman},
 PUBLISHER = {Benjamin/Cummings Publishing Co. Inc. Advanced Book Program},
   ADDRESS = {Reading, Mass.},
      YEAR = {1978},
     PAGES = {xxii+m-xvi+806},
      ISBN = {0-8053-0102-X},
   MRCLASS = {58Fxx (70-02 70F07)},
  MRNUMBER = {MR515141 (81e:58025)},
MRREVIEWER = {D. L. Rod},
}

@article {ar:lall_west-discrete_variational_hamiltonian_mechanics,
    AUTHOR = {Lall, S. and West, M.},
     TITLE = {Discrete variational {H}amiltonian mechanics},
   JOURNAL = {J. Phys. A},
  FJOURNAL = {Journal of Physics. A. Mathematical and General},
    VOLUME = {39},
      YEAR = {2006},
    NUMBER = {19},
     PAGES = {5509--5519},
      ISSN = {0305-4470},
     CODEN = {JPHAC5},
   MRCLASS = {37J05 (34A26 39A10 49S05 70H20)},
  MRNUMBER = {2220773 (2007h:37080)},
MRREVIEWER = {Juan Carlos Marrero Gonzalez},
       DOI = {10.1088/0305-4470/39/19/S11},
       URL = {http://dx.doi.org/10.1088/0305-4470/39/19/S11},
}

@article {ar:deDiego_deAlmagro-variational_order_for_forced_lagrangian_systems,
    AUTHOR = {Mart\'{\i}n de Diego, D. and Mart\'{\i}n de Almagro, R. T. Sato},
     TITLE = {Variational order for forced {L}agrangian systems},
   JOURNAL = {Nonlinearity},
  FJOURNAL = {Nonlinearity},
    VOLUME = {31},
      YEAR = {2018},
    NUMBER = {8},
     PAGES = {3814--3846},
      ISSN = {0951-7715},
   MRCLASS = {37M15 (65P10 70G75 70H03 70H05)},
  MRNUMBER = {3826116},
MRREVIEWER = {Ram Krishan Sharma},
       DOI = {10.1088/1361-6544/aac5a6},
       URL = {https://doi.org/10.1088/1361-6544/aac5a6},
}

@article {ar:deLeon_lainz_lopezGordon-discrete_hamilton_jacobi_theory_for_systems_with_external_forces,
    AUTHOR = {de Le\'{o}n, Manuel and Lainz, Manuel and L\'{o}pez-Gord\'{o}n, Asier},
     TITLE = {Discrete {H}amilton-{J}acobi theory for systems with external
              forces},
   JOURNAL = {J. Phys. A},
  FJOURNAL = {Journal of Physics. A. Mathematical and Theoretical},
    VOLUME = {55},
      YEAR = {2022},
    NUMBER = {20},
     PAGES = {Paper No. 205201, 30},
      ISSN = {1751-8113},
   MRCLASS = {70S05},
  MRNUMBER = {4418632},
       DOI = {10.1088/1751-8121/ac6240},
       URL = {https://doi.org/10.1088/1751-8121/ac6240},
}

@article {ar:marrero_martin_martinez-discrete_lagrangian_and_hamiltonian_mechanics_on_lie_groupoids,
    AUTHOR = {Marrero, Juan C. and Mart{\'{\i}}n de Diego, David and
              Mart{\'{\i}}nez, Eduardo},
     TITLE = {Discrete {L}agrangian and {H}amiltonian mechanics on {L}ie
              groupoids},
   JOURNAL = {Nonlinearity},
  FJOURNAL = {Nonlinearity},
    VOLUME = {19},
      YEAR = {2006},
    NUMBER = {6},
     PAGES = {1313--1348},
      ISSN = {0951-7715},
     CODEN = {NONLE5},
   MRCLASS = {37J05 (17B66 22A22 53D20 70G45 70H30)},
  MRNUMBER = {2230001 (2007c:37068)},
MRREVIEWER = {Pawe{\l}Urbanski},
       DOI = {10.1088/0951-7715/19/6/006},
       URL = {http://dx.doi.org/10.1088/0951-7715/19/6/006},
}

@article{ar:wendtlandt_marsden-mechanical_integrators_derived_from_a_discrete_variational_principle,
 author = {Wendlandt, Jeffrey M. and Marsden, Jerrold E.},
 title = {Mechanical integrators derived from a discrete variational principle},
 fjournal = {Physica D},
 journal = {Physica D},
 issn = {0167-2789},
 volume = 106,
 number = {3-4},
 pages = {223--246},
 year = 1997,
 language = {english},
 doi = {10.1016/S0167-2789(97)00051-1},
 zbMATH = 1392832,
 Zbl = {0963.70507}
}

\end{document}